\documentclass{amsart}

\usepackage{epsfig}
\usepackage{amsmath}
\usepackage{amssymb}
\usepackage{amscd}
\usepackage{graphicx}
\usepackage{color}
\usepackage{tikz}
\usetikzlibrary{arrows,shapes,positioning,shadows,trees}
\tikzset{
  basic/.style  = {draw, text width=2cm, drop shadow, font=\sffamily, rectangle},
  root/.style   = {basic, rounded corners=2pt, thin, align=center,
                   fill=white!30, text width=12em},
  level 2/.style = {basic, rounded corners=6pt, thin,align=center, fill=white!60,
                   text width=9em},
  level 3/.style = {basic, rounded corners=6pt, thin,align=center, fill=white,
                   text width=9em}
}

\newcommand\mL{L\kern-0.08cm\char39}

\newtheorem{thm}{Theorem}[section]
\newtheorem{lem}[thm]{Lemma}
\newtheorem{prop}[thm]{Proposition}

\newtheorem{cor}[thm]{Corollary}

\newcommand{\diam}{{\mathrm{diam}}}
\newcommand{\dist}{{\mathrm{dist}}}
\newcommand{\Tran}{{\mathrm{Tran}}}
\newcommand{\orb}{{\mathrm{orb}}}

\newcommand{\sep}{\mathop{\rm sep}}
\newcommand{\Prox}{{\mathrm{Prox}}}

\numberwithin{equation}{section}

\makeatletter
\@namedef{subjclassname@2010}{%
  \textup{2010} Mathematics Subject Classification}
\makeatother

\begin{document}

\title{Dynamical compactness and sensitivity}

\subjclass[2010]{}

\author{Wen Huang, Danylo Khilko, Sergi{\u\i} Kolyada and Guohua Zhang}

\address{Department of Mathematics, Sichuan University,
Chengdu, Sichuan 610064, China}

\address{School of Mathematical Sciences, University of Science and Technology of
China, Hefei, Anhui 230026, China}

\email{wenh@mail.ustc.edu.cn}

\address{Faculty of Mechanics and Mathematics, National Taras Shevchenko University of Kyiv,
Academician Glushkov prospectus 4-b, 03127 Kyiv, Ukraine}

\email{dkhilko@ukr.net}

\address{Institute of Mathematics, NASU, Tereshchenkivs'ka 3, 01601 Kyiv, Ukraine}

\email{skolyada@imath.kiev.ua}

\address{School of Mathematical Sciences and LMNS, Fudan University and Shanghai Center for Mathematical Sciences, Shanghai 200433, China}

\email{chiaths.zhang@gmail.com}

\begin{abstract}
To link the Auslander point dynamics property with topological transitivity, in this paper we introduce dynamically compact systems as a new concept of a chaotic
dynamical system $(X,T)$ given by a compact metric space $X$ and a continuous
surjective self-map $T:X \to X$. Observe that each weakly mixing system is transitive compact, and we show that any transitive compact M-system is weakly mixing.
Then we discuss the relationships among it and other several stronger forms of sensitivity.
We prove that
any transitive compact system is Li-Yorke sensitive and furthermore multi-sensitive if it is not proximal,
and that any multi-sensitive system has positive topological sequence entropy. Moreover, we show that
multi-sensitivity is equivalent to both thick sensitivity and thickly syndetic sensitivity for M-systems.
We also give a quantitative analysis for multi-sensitivity of a dynamical system.
\end{abstract}



\dedicatory{This paper is dedicated to Professor \mL ubom\'\i r Snoha on the occasion of his 60th birthday.}

\maketitle

\markboth{}{}


\section{Introduction}

By a (\emph{topological}) \emph{dynamical system} we mean a compact metric space with a continuous self-surjection. We say it \emph{trivial} if in addition the space is a singleton.
\emph{Throughout this paper, we are only interested in a nontrivial dynamical system, where the state space is a compact metric space without isolated
points,} and let $(X,T)$ be such a dynamical system with metric $d$.


Let $\mathbb{Z}_{+}$ be the set of all nonnegative integers and $\mathbb{N}$ the set of all positive integers.
Recall that the system $(X, T)$ is (\textit{topologically}) \emph{transitive} if $N_T (U_1, U_2)= \{n\in \mathbb{Z}_+: U_1\cap T^{-n}U_2\neq \varnothing \}$
 is nonempty for any \emph{opene}\footnote{Because we so often have to refer to
open, nonempty subsets, we will call such subsets \emph{opene}.} subsets $U_1, U_2\subset X$, equivalently, $N_T (U_1, U_2)$
 is infinite for any opene subsets $U_1, U_2\subset X$.

H. Furstenberg started a systematic study of transitive dynamical
systems in his paper on disjointness in topological dynamics and ergodic
theory \cite{Furstenberg1967}, and the theory was further developed in \cite{FW1978} and \cite{Furstenberg1981}.
The main motivation for this paper comes from \cite{Sn1}, \cite{Akin1997}, \cite{JiSn2003}, \cite{AK}, \cite{Glasner2004}, \cite{HSY2004}, \cite{ShYe2004}, \cite{Shao2006}
and recent papers \cite{Li}, \cite{OZ} and \cite{LiYe}, which discusses a dynamical property called transitive compactness examined firstly for weakly mixing systems in \cite{AK}: transitive compactness is quite related to but different from the property of transitivity, it will be equivalent to weak mixing under some weak conditions, and it presents some kind of sensitivity of the system.


Before going on, let us recall some notions and facts from topological dynamics.
Denote by $\mathcal{P} =
\mathcal{P}(\mathbb{Z}_{+})$ the set of all subsets of $\mathbb{Z}_{+}$.
A subset $\mathcal{F}$ of $\mathcal{P}$  is a (\emph{Furstenberg}) \emph{family}, if it
is hereditary upward, that is, $F_1 \subset F_2$ and $F_1 \in \mathcal{F}$ imply $F_2 \in \mathcal{F}$.
For a family $\mathcal{F}$,
the \emph{dual family} of $\mathcal{F}$,
denoted by $k\mathcal{F}$, is defined as $$\{F\in \mathcal{P}: F\cap F'\ne \varnothing \mbox{ for any }
F' \in \mathcal{F}\}.$$
Denote by $\mathcal{B}$ the family of all infinite
subsets of $\mathbb{Z}_{+}$.
Let $F\in \mathcal{P}$. Recall that a subset $F$ is \emph{thick} if it contains arbitrarily long runs of positive integers.
Each element of the dual family of the thick family is said to be \emph{syndetic}, equivalently, $F$ is syndetic if and only if there is $N\in \mathbb{N}$ such that $\{i, i + 1, \dots , i + N \} \cap F \not = \varnothing$
for every $i \in \mathbb{Z}_{+}$. We say that $F$ is \emph{thickly syndetic} if for every $N\in \mathbb{N}$ the positions where
length $N$ runs begin form a syndetic set.

For any $F\in \mathcal{P}$ we will define $T^{F}x= \{ T^{i}x : i \in F \}.$ Let us also define
$ N_T(x,G)=\{n\in \mathbb{Z}_{+}: T^nx \in G \}$ for every point $x\in X$
and each subset $G$ of $X$. Let $\mathcal{F}$ be a family and $x\in X$, \emph{the  $\omega$-limit set of point $x$  with respect to  the family $\mathcal{F}$} (see \cite{Akin1997}), or shortly \emph{the  $\omega_\mathcal{F}$-limit set of point $x$}, denoted by $\omega_{\mathcal{F}}(x)$, is defined as
$$\bigcap_{F \in \mathcal{F}} \overline{T^{F}x}
= \{ z\in X: N_T(x,G) \in k\mathcal{F} \ \mbox{for every neighborhood} \ G \ \mbox{of} \ z \}.$$
Recall that \emph{the $\omega$-limit set of $x$}, denoted by $\omega_{T}(x)$, is defined as
$$\bigcap_{n=1}^\infty \overline{\{T^kx: k\ge n\}}
= \{ z\in X: N_T(x,G) \in \mathcal{B}  \ \mbox{for every neighborhood} \ G \ \mbox{of} \ z \}. $$
In particular, $\omega_{\mathcal{F}}(x)$ is a subset of $\omega_{T}(x)$.

\medskip

We say that the system $(X,T)$ is \emph{compact with respect to the family
$\mathcal{F}$}, or shortly
\emph{dynamically compact}, if the $\omega_\mathcal{F}$-limit set
$\omega_{\mathcal{F}}(x)$ is nonempty for all $x\in X$.
In this paper we consider one of such dynamical compactness --- transitive compactness, and its
relations with well-known chaotic properties of
dynamical systems.

Recall that a pair of points $x, y \in X$ is \emph{proximal} if $\liminf_{n\rightarrow \infty} d (T^n x, T^n y)= 0$.
In this case each of points from the pair is said to be \emph{proximal} to another. Denote by $\Prox _T (X)$ the set of all proximal pairs of points.
For each $x\in X$, denote by $\Prox _T (x)$, called the \emph{proximal cell of $x$}, the set of all points
which are proximal to $x$.  It is a deep theorem of Auslander that every proximal cell contains a minimal point (i.e. point
from a minimal set) \cite{Auslander1988}.

The following property of dynamical systems, which was first examined in \cite{AK} for topologically weakly mixing
systems, links the Auslander point dynamics
property with topological transitivity. Let $\mathcal{N}_T$ be the set of all subsets of $\mathbb{Z}_+$ containing some $N_T(U,V)$, where $U,V$ are opene subsets of $X$.
A dynamical system $(X, T)$ is called
\emph{transitive compact}, if for any point $x\in X$ the $\omega_{\mathcal{N}_T}$-limit set
$\omega_{\mathcal{N}_T}(x)$ is nonempty,
in other words,
for any point $x\in X$ there exists
a point $z \in X$ such that $$ N_T(x,G) \cap N_T(U,V)\ne \varnothing$$
 for any  neighborhood
$G$ of $z$ and any opene subsets $U,V$ of $X$.

Obviously that
a transitive compact system is topologically transitive.
What is the difference between a compact, topologically transitive system and a dynamically
compact, topologically transitive system? Roughly speaking the first system is of the form  $(X,T, \mathbb{Z}_+)$
and the second one --- $(X,T,\mathcal{F})$.
 In the second system we dynamically changed the time for possible
behaviour of the orbits (by using $\mathcal{N}_T$). The area of dynamical systems where one investigates dynamical properties that can be
described in topological terms is called topological dynamics. Investigating the topological
properties of maps that can be described in dynamical terms is in a sense the
opposite idea. This area is called \emph{dynamical topology} (see \cite{KMS2014}). Some results of this paper can be considered as a
contribution to dynamical topology.

We will show that the $\omega_{\mathcal{N}_T}$-limit sets are positively invariant closed subsets of the system (Lemma \ref{0200}) and in fact invariant closed subsets if the system is weakly mixing (Proposition \ref{201509042231}).
Even though any transitive compact M-system is weakly mixing
(Corollary \ref{201509061135}) and observe that each weakly mixing system is transitive compact, there exist non totally transitive (and hence non weakly mixing), transitive compact systems in both proximal and non-proximal cases (Theorem \ref{0206}).


As shown by Theorem \ref{0208} and Theorem \ref{0206a}, a transitive compact system presents some kinds of sensitivity.
The notion of sensitivity (sensitive dependence on initial conditions) was first used by Ruelle \cite{Ru}, which captures the idea that in a chaotic system a small change in the initial condition can cause a big change in the trajectory.
Since then, many authors studied different properties related
to sensitivity (cf. \cite{GlasnerMaon1989}, \cite{AAB1993}, \cite{GlasnerWeiss1993}, \cite{Akin1997}, \cite{AkinGlasner2001}, \cite{AK}, \cite{ABC}, \cite{HLY}, \cite{LiTuYe}). For the recent development of sensitivity in topological dynamics see for example the survey \cite{LiYe} by Li and Ye.

According to the works by Guckenheimer \cite{Gu}, Auslander and Yorke \cite{AuslanderYorke} a dynamical
system $(X,T)$ is
{\it sensitive} if there exists $\delta> 0$ such that for every $x\in X$ and every
neighborhood $U_x$ of $x$, there exist $y\in U_x$ and $n\in \mathbb{N}$ with $d(T^n x,T^n y)> \delta$.
Such a $\delta$ is also called a \emph{sensitive constant} of the system $(X, T)$.
Recently in \cite{Subrahmonian2007} Moothathu initiated a way to measure the sensitivity of a system, by checking how large is the set of
nonnegative integers for which the sensitivity still occurs. Following \cite{Subrahmonian2007} and \cite{liuheng}, we recall the
definitions of some stronger versions of sensitivity.

Let $\delta> 0$. For any opene set $U\subset X$ we define
$$ S_T (U, \delta)= \{n\in \mathbb{Z}_+: \mbox { there are } x_1, x_2\in U\ \mbox{such that}
\ d (T^n x_1, T^n x_2)>
\delta\}.$$
A dynamical system $(X,T)$ is \emph{thickly sensitive} (\emph{thickly syndetically sensitive}, respectively) if there exists $\delta> 0$ such that
$S_T (U, \delta)$
is thick (thickly syndetic, respectively) for any opene $U\subset X$;
and is \emph{multi-sensitive}
if there exists $\delta> 0$ such that $\bigcap_{i= 1}^k S_T (U_i, \delta)\neq \varnothing$ for any finite collection of
opene $U_1, \dots, U_k\subset X$.
Remark that: a non-minimal M-system is thickly syndetically sensitive \cite[Theorem 8]{liuheng}, and any thickly syndetically sensitive system is multi-sensitive because the intersection of finitely many thickly syndetically sets is also thickly syndetic; furthermore, the first, third and fourth authors of the present paper proved that any multi-sensitive system is thickly sensitive and these two sensitivities are equivalent for transitive systems \cite[Proposition 3.2]{HKZ}.

In this paper, we shall show that: multi-sensitivity is equivalent to both thick sensitivity and thickly syndetic sensitivity for M-systems (Theorem \ref{1407090448}), any multi-sensitive system has positive topological sequence entropy (Proposition \ref{1407091722}) and any transitive compact system is Li-Yorke sensitive (Theorem \ref{0208}) and furthermore multi-sensitive if it is not proximal (Theorem \ref{0206a}). Recall that we have assumed that we are only interested in nontrivial dynamical systems.


Observe that in \cite{Kolyada2013} Rybak and the third author of the paper initiated another way to measure the sensitivity of a
system, that is, gave quantitative measures of the sensitivity of a dynamical system by introducing the Lyapunov numbers:
\begin{eqnarray*}
\mathbb{L}_{r} & = & \sup\{\delta: \mbox{ for every } x\in X
\mbox{ and every open neighborhood } U_x \mbox{ of } x \mbox{ there } \\ & & \hskip 35pt \mbox{ exist } y\in U_x  \mbox{ and a nonnegative integer }
n  \mbox{ with } \\ & & \hskip 35pt d (T^n x,T^n y)> \delta
 \};
 \end{eqnarray*}
 \begin{eqnarray*}
\overline{\mathbb{L}}_{r} &=& \sup\{\delta: \mbox{ for every } x\in X \mbox{ and every open neighborhood }
U_x \mbox{ of } x \mbox { there } \\ & & \hskip 35pt
\mbox{exists } y\in U_x \mbox{ with }
\limsup_{n\to \infty}d (T^n x,T^n y)> \delta\}; \\
\mathbb{L}_{d} & = & \sup\{\delta: \mbox{ in any opene }
U\subset X \mbox{ there exist } x,y\in U  \mbox{ and a nonnegative} \\
& & \hskip 35pt \mbox {integer } n \mbox{ with }  d(T^n x,T^n y)>
\delta \}; \\
\overline{\mathbb{L}}_{d}  &=& \sup\{\delta: \mbox{ in any
opene } U\subset X \mbox{ there exist } x,y\in U
\mbox{ with } \\ & & \hskip 35pt
\limsup_{n\to \infty}d(T^n x,T^n y)> \delta \}.
\end{eqnarray*}
Here we set $\sup \varnothing = 0$ by convention.
 Various definitions of sensitivity, formally give us different Lyapunov numbers. Nevertheless, as was shown in \cite{Kolyada2013}, for minimal topologically weakly mixing systems
all these Lyapunov numbers are the same.

The motivation of \cite{Kolyada2013} comes from the
following proposition according to \cite{AK}:

\begin{prop} \label{prop}
The following conditions are
equivalent\emph{:}
\begin{itemize}

\item[1.] $(X,T)$ is sensitive.

\item[2.] There exists $\delta> 0$ such that for every $x\in X$ and every
neighborhood $U_x$ of $x$, there exists $y\in U_x$ with $\limsup_{n\to \infty}d(T^n x,T^n y)> \delta$.

\item[3.] There exists $\delta> 0$ such that in any opene
$U$ in $X$ there are $x, y\in U$ and a nonnegative integer $n$ with $d(T^n x,T^n y)> \delta$.

\item[4.] There exists $\delta> 0$ such that in any opene $U\subset X$ there are $x,y\in U$ with $\limsup_{n\to \infty}d(T^n x,T^n y)> \delta$.
\end{itemize}
\end{prop}

Note that an analogue of Proposition \ref{prop} can be obtained for multi-sensitive systems (Proposition \ref{multi-proposition}). Thus following the line in \cite{Kolyada2013}, we can give quantitative measures for the multi-sensitivity of a dynamical system. Precisely, we introduce the following several quantities (we still call them \emph{Lyapunov numbers}):
\begin{eqnarray*}
\mathbb{L}_{m,r}&= & \sup \{\delta:  \text{ for any finite collection } x_1, \dots, x_k\ \text{ of points in $X$ and }\\ & & \hskip 35pt
\text{
any system of open neighborhoods } U_i\ni x_i\ (i= 1, \dots, k); \\ & & \hskip 35pt
\text{there exist points }
y_i\in U_i\ \text {and a nonnegative integer } n \\ & & \hskip 35pt
\text{with }\
\min_{1\le i\le k} d (T^n x_i, T^n y_i)> \delta  \};\\
\overline{\mathbb{L}}_{m,r}&= & \sup \{\delta:  \text{ for any finite collection } x_1, \dots, x_k\ \text{of points in $X$ and } \\ & & \hskip 35pt
\text{any system of open neighborhoods } U_i\ni x_i\ (i= 1, \dots, k); \\ & & \hskip 35pt
\text{there exist points } y_i\in U_i\
\text{with } \\ & & \hskip 35pt \limsup_{n\rightarrow \infty}
\min_{1\le i\le k} d (T^n x_i, T^n y_i)> \delta  \};\\
\mathbb{L}_{m,d}&= & \sup \{\delta:  \text{ for any finite collection $U_1, \dots, U_k$ of opene subsets of $X$,} \\
 & &\hskip 35pt \text{there exist } x_i, y_i\in U_i \text { and  a nonnegative integer } n  \text{  with } \\ & & \hskip 35pt  \min_{1\le i\le k} d (T^n x_i, T^n y_i)> \delta \};
 \\
\overline{\mathbb{L}}_{m, d} &= & \sup \{\delta: \text{ for any finite collection $U_1, \dots, U_k$ of opene subsets of $X$,} \\
 & &\hskip 35pt \text{there exist } x_i, y_i\in U_i \text{ with } \limsup_{n\rightarrow \infty} \min_{1\le i\le k} d (T^n x_i, T^n y_i)> \delta \}.
\end{eqnarray*}
It is not striking that these new Lyapunov numbers $\mathbb{L}_{m,r},
\overline{\mathbb{L}}_{m,r}, \mathbb{L}_{m,d}$ and $\overline{\mathbb{L}}_{m, d}$ are all related to each other.
In particular, we prove that:  $\mathbb{L}_{m,r}= \overline{\mathbb{L}}_{m,r}$ for a system with a dense set of
 distal points (Proposition \ref{1407052219}); $\mathbb{L}_{m,d}= \overline{\mathbb{L}}_{m,d}$ for a transitive system (Lemma \ref{1407241808}) and $\mathbb{L}_{m,d}= \overline{\mathbb{L}}_{m,d}= \diam (X)$
 and $\mathbb{L}_{m,r}= \overline{\mathbb{L}}_{m,r}> 0$ for a nontrivial weakly mixing system (Proposition \ref{1407062303}).


The paper is organized as follows. In section 2 we recall from topological dynamics some basic concepts and properties used in later discussions. In section 3
we study transitive compact systems. Observe that any weakly mixing system is transitive compact, we will show that any transitive compact M-system is weakly mixing (Corollary \ref{201509061135}), and provide non totally transitive (and hence non weakly mixing), transitive compact systems in both proximal and non-proximal cases (Theorem \ref{0206}).
In section 4 we explore the implications among transitive compactness and other various stronger forms of sensitivity, in particular, we show that
 multi-sensitivity is equivalent to both thick sensitivity and thickly syndetic sensitivity for M-systems (Theorem \ref{1407090448}), and that any transitive compact system is Li-Yorke sensitive (Theorem \ref{0208}) and furthermore multi-sensitive if it is not proximal (Theorem \ref{0206a}).
 In section 5 we carry out quantitative analysis for multi-sensitivity by studying the relationships between the above introduced four new Lyapunov numbers.
In section 6 we explore further study for multi-sensitivity, and prove that any multi-sensitive system has positive topological sequence entropy (Proposition \ref{1407091722}).
In the last section we consider some related properties and open questions.
\vskip 10pt

\noindent {\bf Acknowledgements.}  The first, third and fourth authors acknowledge the hospitality
of the Max-Planck-Institute f\"ur Mathematik (MPIM) in Bonn, where a
substantial part of this paper was written during the Activity ``Dynamics and Numbers'', June -- July 2014.
We thank MPIM for providing an ideal setting. We also thank Xiangdong Ye
for sharing his joint work \cite{LiYe, YeYu}, and thank the referee for comments that
have resulted in substantial improvements to this paper.

The first author was supported by NNSF of China (11225105, 11431012), the fourth author
was supported by NNSF of China (11271078).

\section{Preliminaries}

In this section we recall standard concepts and results used in later discussions.

\subsection{Transitivity, mixing and minimality}

 A point $x\in X$ is called \emph{fixed} if $T x= x$, \emph{periodic} if $T^n x= x$ for some $n\in \mathbb{N}$, and \emph{transitive} if its \emph{orbit} $\orb_T (x)=
   \{T^n x: n= 0, 1, 2, ... \}$ is dense in $X$. Denote by $\Tran (X, T)$ the set of all transitive
    points of $(X, T)$. Since $T$ is surjective, the system $(X, T)$ is transitive if and only if $\Tran (X, T)$ is a dense $G_\delta$ subset of $X$.

    The system $(X, T)$ is called \emph{minimal} if $\Tran(X,T) = X$.
    In general, a subset $A$ of $X$ is \emph{$T$-invariant} if $TA = A$, and \emph{positively $T$-invariant} if $T A\subset A$.
    If $A$ is a closed, nonempty, $T$-invariant subset then $(A,T|_A)$ is called the associated \emph{subsystem}.
    A \emph{minimal subset} of $X$ is a nonempty, closed, $T$-invariant subset such that the associated subsystem is minimal.
    Clearly, $(X,T)$ is minimal if and only if it admits no a proper, closed, nonempty, positively $T$-invariant subset.
     A point $x \in X$ is called \emph{minimal} if it lies in some minimal subset. In this case, in order to emphasize the underlying system $(X, T)$ we also say that $x\in X$ is a \emph{minimal point of $(X, T)$}.
     Zorn's Lemma implies that every closed, nonempty, positively $T$-invariant set contains a minimal set.
Note that by the classic result of Gottschalk a point
 $x\in X$ is minimal if and only if $N_T (x, U)= \{n\in \mathbb{Z}_+: T^n x\in U\}$ is syndetic for any neighborhood $U$ of $x$, and
 by \cite[Corollary 3.2]{Ye1992} for any $k\in \mathbb{N}$ a point is
   a minimal point of $(X, T)$ if and only if it is a minimal point of $(X, T^k)$, here $T^k$ denotes the composition map of $k$ copies of the map $T$.

     An \emph{M-system} is a transitive system
      with a dense set of minimal points \cite{GlasnerWeiss1993}.
  Let $(X, T)$ and $(Y, S)$ be topological dynamical systems and $k\in \mathbb{N}$. The product system $(X\times Y, T\times S)$ is defined naturally, and denote by $(X^k, T^{(k)})$ the product system of $k$ copies of the system $(X, T)$ for each $k\in \mathbb{N}$.

     Recall that the system $(X, T)$ is \emph{totally transitive} if $(X, T^k)$ is transitive for each $k\in \mathbb{N}$; is
(\textit{topologically}) \emph{weakly mixing} if the product system $(X\times X, T\times T)$ is
 transitive; and is \emph{topologically mixing} if $N_T(U,V)$ is cofinite for any opene sets $U,V$ in $X$, where we say that $\mathcal{S}\subset \mathbb{Z}_+$ is \emph{cofinite}
  if $\mathcal{S}\supset \{m, m+ 1, m+ 2, \dots\}$ for some $m\in \mathbb{N}$. Note that the system $(X, T)$ is weakly mixing if and only if
$N_T(U,V)$ is a thick set for any opene sets $U,V$ in $X$ by \cite{Furstenberg1967} and \cite{Petersen}, in particular, any weakly mixing system is totally transitive.

We say that $(X, T)$ is \emph{topologically ergodic} (\emph{thickly syndetically transitive}, respectively) if the set
 $N_T (U, V)$ is syndetic (thickly syndetic, respectively) for any opene $U, V\subset X$. Note that
any minimal system is topologically ergodic, and a dynamical system $(X, T)$
   is thickly syndetically transitive if and only if $(X, T)$ is not only weakly mixing but also topologically ergodic
    \cite[Theorem 4.7]{HY2002}.


We will say that a point
$x\in X$ is \emph{distal} if it is not proximal to any another point from the closure of the orbit $\{T^n x: n\in \mathbb{Z}_+\}$.
Recall that a dynamical system $(X,T)$ is called \emph{proximal}
if $\Prox_T (X)= X^2$, and is \emph{distal} if any point of $X$ is distal. The system $(X,T)$ is
proximal if and only if $(X,T)$ has the unique fixed point, which is the only minimal point of $(X,T)$ (e.g. see \cite{AK}).

\subsection{Equicontinuity and sensitivity}

 A pair of points  $x, y\in X$ is called a \emph{Li-Yorke pair} if
$\liminf_{n\rightarrow \infty} d (T^n x, T^n y)= 0$ while
$\limsup_{n\rightarrow \infty} d (T^n x, T^n y)> 0$, and a dynamical system $(X, T)$ is called \emph{spatio-temporally chaotic} if for any point $x\in X$ and its neighborhood $U_x$
there is a point $y\in U_x$ such that the pair $x, y$ is Li-Yorke \cite{BGKM2002}.
The system $(X,T)$ is \emph{cofinitely sensitive} if there exists $\delta> 0$ such that
$S_T (U, \delta)$
is cofinite for any opene $U\subset X$ (see \cite{Subrahmonian2007}), and is \emph{Li-Yorke sensitive} if there exists $\delta >0$ such that for any point $x\in X$ and its neighborhood $U_x$
there is a point $y\in U_x$ with $\liminf_{n\rightarrow \infty} d (T^n x, T^n y)= 0$ while
$\limsup_{n\rightarrow \infty} d (T^n x, T^n y)> \delta$ (see \cite{AK}). It is clear that Li-Yorke sensitivity is much stronger than both sensitivity
and spatio-temporal chaos.


The Lyapunov stability or, in other words, equicontinuity is the opposite to the
notion of sensitivity.
 A dynamical system $(X, T)$ is \emph{equicontinuous} if $\{T^{n} : n \geq 0 \}$ is
equicontinuous at any point of $X$, equivalently, for every $\epsilon > 0$ there exists a $\delta > 0$ such that $d(x,x') < \delta$
implies $d (T^n x, T^n x')< \epsilon$ for any $n\in \mathbb{N}$. Remark that each dynamical system admits a maximal equicontinuous factor.

Following \cite{HKZ}, recall that $x\in X$ is \emph{syndetically equicontinuous} if for any $\epsilon> 0$ there exist open $U\subset X$
 containing $x$ and a syndetic set $\mathcal{N}\subset \mathbb{N}$ such that $d (T^n x, T^n x')\le \epsilon$ whenever $x'\in U$
 and $n\in \mathcal{N}$.
 Denote by $\text{Eq}_{\text{syn}} (X, T)$ the set of all syndetically equicontinuous points of $(X, T)$.

\subsection{Other concepts}

  Recall that $\mathcal{S}\subset \mathbb{N}$ is an \emph{IP set}
 if there exists $\{p_k: k\in \mathbb{N}\}\subset \mathbb{N}$ with $\{p_{i_1}+ \dots+ p_{i_k}: k\in \mathbb{N}\
 \text{and}\ i_1< \dots< i_k\}\subset \mathcal{S}$. Denote by $\mathcal{F}_{ip}$ the family of all IP sets. Notice that for an IP set $\mathcal{S}$, $\mathcal{S}= \mathcal{S}_1\cup
   \mathcal{S}_2$ implies that either $\mathcal{S}_1$ or $\mathcal{S}_2$ is an IP set by Hindman's theorem (see for
    example \cite[Theorem 8.12]{Furstenberg1981}).

Recall that $\mathcal{S}\subset \mathbb{N}$ is an \emph{IP$^*$ set} if $\mathcal{S}\cap \mathcal{T}\neq \varnothing $
  for each IP set $\mathcal{T}\subset \mathbb{N}$. It is easy to see that the intersection of an IP set and an IP$^*$ set
   is an infinite set.
Note that by \cite[Theorem 9.11]{Furstenberg1981}: $x\in X$
 is distal if and only if $N_T (x, U)$ is an IP$^*$ set for any neighborhood $U$ of $x$; and for distal points $x_i\in X_i$ of the system $(X_i, T_i), i= 1, \dots, k$, the point $(x_1, \dots, x_k)\in
  X_1\times \dots\times X_k$ is also
 a distal point of the system $(X_1\times \dots\times X_k, T_1\times \dots\times T_k)$.


      Let $(X, T)$ and $(Y, S)$ be topological dynamical systems. Recall that by a \emph{factor map} $\pi: (X, T)\rightarrow (Y, S)$ we mean that $\pi: X\rightarrow Y$ is a continuous
 surjection with $\pi\circ T= S\circ \pi$. In this case, we call $\pi: (X, T)\rightarrow (Y, S)$
    an \emph{extension}; and $(X, T)$  an \emph{extension} of
  $(Y, S)$, $(Y, S)$ a \emph{factor} of $(X, T)$. If, additionally, $\pi: X\rightarrow Y$ is almost one-to-one, that is, there exists a dense subset $Y_0\subset Y$
such that $\pi^{- 1} (y)$ is a singleton for each $y\in Y_0$, then we also call $(X, T)$ an \emph{almost one-to-one extension of $(Y, S)$}.
It is easy to show that the transitive compactness property of a dynamical system
 $(X,T)$ is preserved by a factor map.


Let $A$ be a nonempty finite set. We call $A$ the alphabet
and elements of $A$ are symbols. The \emph{full \emph{(}one-side\emph{)} $A$-shift} is defined as
$$\Sigma= \{x = \{x_i\}_{i=0}^\infty: x_i \in A \text{ for all } i \in \mathbb{Z}_+\}.$$
We equip $A$ with the discrete topology and $\Sigma$ with the product topology.
Usually we write an element of $\Sigma$ as $x = \{x_i\}_{i=0}^\infty = x_0 x_1 x_2 x_3 \dots$ The \emph{shift
map} $\sigma : \Sigma \to \Sigma$  is a continuous map given by
$$ x = \{x_i\}_{i=0}^\infty \mapsto \sigma x = \{x_{i+1}\}_{i=0}^\infty .$$
That is, $\sigma (x)$ is the sequence obtained by dropping the first symbol of $x$. A full binary shift is the full $2$-shift.

A block over $\Sigma$ is a finite sequence of symbols and its length
is the number of its symbols. An $n$-block stands for a block of length $n$. The
set of all blocks over $\Sigma$  is denoted by $\Sigma^*$. The concatenation of two
blocks $u = a_1 \dots a_k$ and $v = b_1 \dots b_l$ is the block $uv = a_1 \dots a_k b_1 \dots b_l$. We write
$u^n$ for the concatenation of $n\geq 1$ copies of a block $u$ and $u^\infty$ for the sequence $uuu \dots \in  \Sigma.$ By $x_{[i,j]}$ we denote the block $x_i x_{i+1} \dots x_j$, where $0 \leq i \leq j$ and $x =
\{x_k\}_{k=0}^\infty \in \Sigma$.  $X\subset \Sigma$ is called a \emph{subshift} if it
is a nonempty, closed, $\sigma$-invariant subset of $\Sigma$. A cylinder of an $n$-block $w \in \Sigma^*$
in a subshift  $X$ is the set $C[w] = \{x \in X: x_{[0, n-1]} = w\}$. The collection of all cylinders  is a basis of the topology of
$X$.


Now let us recall
the definition of topological sequence entropy for a dynamical system $(X,T)$ by using  the classical Bowen-Dinaburg definition of topological entropy.
Consider an increasing sequence $\mathcal{N} = n_1< n_2< \dots$ of $\mathbb{N}$ and define $n_0=0$ by convention. For any $k\in \mathbb{N}$ the function
$
d_{k}(x,y) =\max_{0\le j\le k-1} d (T^{n_j}x,T^{n_j}y)
$ defines a metric on $X$ equivalent to $d$.
Now fix an integer $k\geq 1$ and $\epsilon >0$. A subset $E\subset X$ is called \emph{$(k,T,\epsilon)$-separated\/} (with respect
to $\mathcal{N}$), if
for any two distinct points $x,y\in E$, $d_{k}(x,y)>\epsilon$. Denote by $\sep(k,T,\epsilon)$ the maximal cardinality of a
$(k,T,\epsilon)$-separated set in $X$ and $
h_\mathcal{N} (T, \epsilon)=\limsup_{k\to\infty}
\frac{1}{k}\log \sep (k,T,\epsilon)$.
 It is obvious that $h_\mathcal{N} (T, \epsilon_1) \geq h_\mathcal{N} (T, \epsilon_2)$ when $\epsilon_1  < \epsilon_2$.
\emph{The topological sequence entropy} of $(X,T)$ along the sequence $\mathcal{N}$ is
defined by
$$
h_\mathcal{N} (T)=\lim_{\epsilon\to 0}\limsup_{k\to\infty}
\frac{1}{k}\log \sep (k,T,\epsilon).
$$
We also call it the \emph{topological entropy} of $(X, T)$ in the case of $\mathcal{N}= \{1, 2, \dots\}$.

\section {Transitive compactness}

In this section we study basic properties of transitive compact systems. We will show that the $\omega_{\mathcal{N}_T}$-limit sets are positively invariant closed subsets and in fact invariant closed subsets if the system is weakly mixing.
Even though any transitive compact M-system is weakly mixing and observe that each weakly mixing system is transitive compact, there exist non totally transitive (and hence non weakly mixing), transitive compact systems in both proximal and non-proximal cases.


Recall that a dynamical system $(X, T)$ is transitive compact, if
for any point $x\in X$ there exists
a point $y \in X$ such that $ N_T(x,G_y) \cap N_T(U,V)\ne \varnothing$
 for any  neighborhood
$G_y$ of $y$ and any opene subsets $U,V$ of $X$.
Note that if $(X, T)$ is topologically transitive then $N_T (U_1, U_2)$
 is infinite for any opene subsets $U_1, U_2\subset X$. Similarly, we can prove that in the above definition of transitive compactness the set
$N_T(x,G_y) \cap N_T(U,V)$ is always infinite.

\begin{lem} \label{02000}
Let $(X,T)$ be a transitive compact system and $x \in X$.  Then the set $N_T(x, G_y) \cap N_T(U, V)$ is infinite for any $y \in \omega_{\mathcal{N}_T}(x)$, neighbourhood $G_y$ of $y$ and opene sets $U, V$ in $X$.
\end{lem}
\begin{proof}
Assume the contrary that there exist opene sets $U, V$ in $X$ and a point $z\in \omega_{\mathcal{N}_T}(x)$ with a neighborhood $G_z$ of $z$ such that $N_T (x, G_z) \cap N_T(U, V)$ is finite, say $N_T(x, G_z) \cap N_T(U, V)= \{n_1, \dots, n_k\}$. As the non-singleton space $X$ contains no isolated points, we can take an opene subset $U_1\subset U$ small enough such that
$V_1:= V\setminus \bigcup_{i= 1}^k T^{n_i} \overline{U_1}$ is an opene subset of $X$. By the construction we have $N_T(U_1, V_1)\subset N_T (U, V)$ and then $N_T(x, G_z) \cap N_T(U_1, V_1)= \varnothing$, a contradiction to the selection of the point $z\in \omega_{\mathcal{N}_T}(x)$. This finishes the proof.
\end{proof}

\begin{lem} \label{02000a}
The following conditions are
equivalent:
\begin{enumerate}
\item[(1)] $(X, T)$ is transitive compact.

\item[(2)] $(X, T)$ is topologically transitive and for any point $x\in X$ there exists
a point $z \in X$, such that $N_T(x,G) \cap N_T(W,T^{-k}W)\neq \varnothing$ for any neighborhood
$G$ of $z$, opene $W$ in $X$ and $k\in \mathbb{Z}_+$.
\end{enumerate}
\end{lem}
\begin{proof}
(1) $\Rightarrow$ (2) is obvious. Just take $U=W$, $V=T^{-k}W.$

(2) $\Rightarrow$ (1) Fix a point $x \in X$. Let $z$ be a point of $X$ such that $ N_T(x,G) \cap N_T(W,T^{-k}W)\neq \varnothing$ for any neighborhood
$G$ of $z$, opene $W$ in $X$ and $k\in \mathbb{Z}_+$. We are
going to show $z \in \omega_{\mathcal{N}_T}(x)$. Let $U$ and $V$ be opene sets in $X$. Since $(X, T)$ is transitive,
there exists $m\in \mathbb{Z}_+$ such that $T^{m}V \cap U \neq \varnothing$. Now let $W=V \cap T^{-m}U\neq \varnothing$. Then
$ N_T(x,G) \cap N_T(W,T^{-m}W)\neq \varnothing,$
therefore there exists $s\in \mathbb{Z}_+$ such that $T^sx \in G$ and $T^sW \cap T^{-m}W \neq \varnothing$.
Hence $\varnothing \neq T^{s+m}W \cap W=T^{s+m}(V \cap T^{-m}U) \cap V \cap T^{-m}U \subset T^s U \cap V$ and we are done.
\end{proof}

Notice that for topological transitive systems the $\omega$-limit set of any point is either the whole space or a nowhere dense set
in the space. In the following we will obtain a similar result for the $\omega_{\mathcal{N}_T}$-limit sets in a transitive compact system $(X, T)$, which comes from the following general fact about the structure of the $\omega_{\mathcal{N}_T}$-limit sets.

\begin{lem} \label{0200}
$\omega_{\mathcal{N}_T}(x)$ is a positively $T$-invariant
closed subset of $X$ for any $x\in X$.
\end{lem}

\begin{proof}
Let $y \in \omega_{\mathcal{N}_T}(x)$ and $G_{Ty}$ be a neighborhood of $Ty$. For each $N_T(U, V) \in \mathcal{N}_T$, as $T^{-1}(G_{Ty})$ is a neighborhood of $y$ and $N_T(U, T^{- 1} V) \in \mathcal{N}_T$, we may take $n\in N_T (x, T^{-1} G_{Ty})\cap N_T (U, T^{-1} V)$ by the assumption of $y \in \omega_{\mathcal{N}_T}(x)$, and then $n+1\in N_T(x, G_{Ty}) \cap N_T(U, V)$. By the arbitrariness of $N_T(U, V)$ and $G_{Ty}$ we have $T y\in \omega_{\mathcal{N}_T}(x)$. Now let $z \in \overline{\omega_{\mathcal{N}_T}(x)}$. In any open neighborhood $G_z$ of the point $z$ we can find a point $y \in \omega_{\mathcal{N}_T}(x)$, and then $N_T (x, G_z) \cap N_T(U, V)\ne \varnothing$ as $G_z$ is also a neighborhood of $y$. This shows $z \in \omega_{\mathcal{N}_T}(x)$.
Summing up, $\omega_{\mathcal{N}_T}(x)$ is a positively $T$-invariant
closed subset of $X$.
\end{proof}

Applying Lemma \ref{0200} we have directly:

\begin{cor} \label{201509031402}
Let $(X,T)$ be a transitive compact system and $x \in X$. Then:
\begin{enumerate}

\item[(1)] either $\omega_{\mathcal{N}_T} (x)=X$ or $\omega_{\mathcal{N}_T}(x)$ is nowhere dense in $X$.

\item[(2)] if $x$ is a minimal point, say $M$ to be the minimal subset containing $x$, then $\omega_{\mathcal{N}_T} (x)=\omega_T (x)= M$.
\end{enumerate}\end{cor}

\begin{proof}
(1) The conclusion follows from the facts that each transitive compact system is transitive, and that any positively $T$-invariant closed subset in a transitive system is either the whole space or a nowhere dense subset (see \cite{KS1997}).

(2) The conclusion follows from Lemma \ref{0200} and the fact of $\varnothing\ne \omega_{\mathcal{N}_T} (x)\subset \omega_T (x)\subset M$ and the minimality of the subset $M$.
\end{proof}

It is not hard to show that the $\omega$-limit set of a system $(X, T)$ is not only positively $T$-invariant but also $T$-invariant.
From Lemma \ref{0200} it seems reasonable to expect that each $\omega_{\mathcal{N}_T}$-limit set is more than positively $T$-invariant, that is, is also $T$-invariant. The answer to the question stands open for the general case, while we can prove:

\begin{prop} \label{201509042231}
Let $(X,T)$ be a weakly mixing system and $x \in X$. Then the subset $\omega_{\mathcal{N}_T}(x)$ is $T$-invariant.
\end{prop}

\begin{proof}
It suffices to consider the case of $\omega_{\mathcal{N}_T}(x)\ne \varnothing$ (in fact, there exists only this case for a weakly mixing system as explained later). Let $y\in \omega_{\mathcal{N}_T}(x)$.
By Lemma \ref{0200}, we only need to find some $z\in \omega_{\mathcal{N}_T}(x)$ with $T z= y$.

For each $k\in \mathbb{N}$ we take a finite open cover $\mathcal{U}_k$ of $X$ with $\text{diam} (\mathcal{U}_k)< \frac{1}{k}$ and a neighborhood $W_k$ of $y$ with $\text{diam} (W_k)< \frac{1}{k}$. As $(X, T)$ is weakly mixing, by \cite{Furstenberg1967} there exist opene subsets $U_1, V_1$ of $X$ such that
$$N_T (U_1, V_1)\subset \bigcap_{U, V\in \mathcal{U}_1} N_T (U, V).$$
As $y\in \omega_{\mathcal{N}_T}(x)$, let $n_1\in N_T (x, W_1)\cap N_T (T^{- 1} U_1, V_1)$. We may assume $n_1\in \mathbb{N}$ by Lemma \ref{02000}.
Now assume that positive integers $n_1< \dots< n_m$ and opene subsets $U_1, V_1, \dots, U_m, V_m$ of $X$ have been constructed, where $m\in \mathbb{N}$, applying \cite{Furstenberg1967} again to weakly mixing $(X, T)$ we can choose opene subsets $U_{m+ 1}, V_{m+ 1}$ of $X$ such that
\begin{equation} \label{201509051737}
N_T (U_{m+ 1}, V_{m+ 1})\subset \bigcap_{U, V\in \mathcal{U}_{m+ 1}} N_T (U, V)\cap \bigcap_{i= 1}^m N_T (U_i, V_i).
\end{equation}
And then applying Lemma \ref{02000} to the assumption of $y\in \omega_{\mathcal{N}_T}(x)$, we can choose $n_{m+ 1}\in N_T (x, W_{m+ 1})\cap N_T (T^{- 1} U_{m+ 1}, V_{m+ 1})$ with $n_m< n_{m+ 1}$.

By choosing a subsequence if necessary we may assume that the sequence $T^{n_i- 1} x$ tends to $z\in X$. From the above construction it is easy to check $T z= y$. Now we will finish the proof by showing $z\in \omega_{\mathcal{N}_T}(x)$. Let $G_z$ be a neighborhood of $z$ and $U', V'$ be opene subsets of $X$. By the above construction, $T^{n_i- 1} x$ belongs to $G_z$ if $i$ is large enough, and we may take opene subsets $U'_t, V'_t\in \mathcal{U}_t$ for some $t\in \mathbb{N}$ such that $U'_t\subset U'$ and $V'_t\subset V'$. We may assume that $t$ is large enough and hence $n_t- 1\in N_T (x, G_z)$. Additionally, $n_t\in N_T (T^{- 1} U_t, V_t)$ and then by \eqref{201509051737} one has
$$n_t- 1\in N_T (U_t, V_t)\subset N_T (U'_t, V'_t)\subset N_T (U', V').$$
In particular, $N_T (x, G_z)\cap N_T (U', V')\neq \varnothing$. This shows $z\in \omega_{\mathcal{N}_T}(x)$ by the arbitrariness of $G_z$ and $U', V'$, which finishes the proof.
\end{proof}


It was first observed in \cite{AK} that each weakly mixing system $(X, T)$ is transitive compact by applying \cite{Furstenberg1967} to $(X, T)$ (and hence there exists a transitive compact, non-minimal, M-system, as there are many weakly mixing, non-minimal, M-systems). The following result extends it a little bit.

\begin{lem} \label{0201}
Let $(X, T)$ be a weakly mixing system, $x\in X$ and $W$ be a closed subset of $X$ such that $N_T(x, W)$ is syndetic. Then
 $\omega_{\mathcal{N}_T}(x) \cap W \neq \varnothing.$
\end{lem}

\begin{proof}
Assume the contrary that $\omega_{\mathcal{N}_T}(x) \cap W=\varnothing$.
Then for every $y \in W$ there exists a neighborhood $G_y$ of the point $y$ and opene subsets
$U_y, V_y$ of $X$ such that
$$N_T (x, G_y) \cap N_T(U_y, V_y)=\varnothing.$$
By the compactness of $W$,
we may choose $y_1, \dots, y_k$ such that $\{G_{y_1}, \dots G_{y_k}\}$ covers $W$.
As the system $(X, T)$ is weak mixing, $\bigcap_{i=1}^{k} N_T(U_{y_i}, V_{y_i})$ is a thick set by \cite{Furstenberg1967}, and then,
by the assumption that $N_T(x, W)$ is syndetic, there exists
$$n \in \bigcap_{i=1}^{k} N_T(U_{y_i}, V_{y_i}) \cap N_T(x, W).$$
In particular, $T^{n}x \in W$ and then $T^{n}x \in G_{y_j}$ for
some $1 \leq j \leq k$. This shows $n\in N_T (x, G_{y_j})\cap N_T(U_{y_j}, V_{y_j})$,
a contradiction to the selection of $G_{y_j}, U_{y_j}, V_{y_j}$.
\end{proof}


In the following we shall prove that the difference between a transitive compact system and a weakly mixing system is not too much.
For example, analogously as it was done for weakly mixing systems in \cite[Theorem 3.8]{AK} one can prove:

\begin{prop} \label{0207}
Let $(X,T)$ be a transitive compact system. Then for every $x\in X$ the proximal cell $\Prox _T(x)$ is a dense $G_\delta$ subset of $X$.
\end{prop}

The following result provides sufficient and necessary conditions for a system being weakly mixing.
Recall that each weakly mixing system is transitive compact.

\begin{prop} \label{0203}
The following conditions are
equivalent:
\begin{enumerate}
\item[(1)] $(X, T)$ is weakly mixing.

\item[(2)] $x\in \omega_{\mathcal{N}_T} (x)$ for each $x\in X'$, where $X'$ is a dense subset of $X$.

\item[(3)] $\omega_{\mathcal{N}_T}(x_0)=X$ for some point $x_0\in X$.

\item[(4)] $\omega_{\mathcal{N}_T}(x)=X$ for each $x\in X''$, where $X''$ is a dense $G_\delta$ subset of
$X$.
\end{enumerate}
\end{prop}

\begin{proof}
(2) $\Rightarrow$ (1) Let $U, V$ be opene subsets of $X$ and we may take a point $x\in U\cap X'$ by the density of $X'$. As $x\in \omega_{\mathcal{N}_T} (x)$, $N_T (x, U)\cap N_T (U, V)\neq \varnothing$ and then $N_T (U, U)\cap N_T (U, V)\neq \varnothing$. By the arbitrariness of $U, V$ we have that $(X, T)$ is weakly mixing by applying \cite[Lemma]{Petersen}.

(3) $\Rightarrow$ (1)
Let $U_1, U_2, V_1, V_2$ be opene subsets in $X$ and take $y_1\in U_1, y_2\in V_1$. Observing
$\omega_{\mathcal{N}_T}(x_0)= X$, in particular, $y_1\in \omega_{\mathcal{N}_T}(x_0)$, there exists $k\in \mathbb{Z}_+$
such that $T^k x_0 \in U_1$, and, in addition, by $y_2\in \omega_{\mathcal{N}_T}(x_0)$ we may choose
 $n \in N_T (x_0, V_1) \cap N_T (T^{-k} U_2, V_2)$ with $n> k$ by Lemma \ref{02000}. In particular,
 $n-k \in N_T(U_1, V_1) \cap N_T(U_2, V_2)$. Then the system $(X,T)$ is
 weakly mixing by the arbitrariness of $U_1, U_2, V_1, V_2$.

(4) $\Rightarrow$ (3) and (4) $\Rightarrow$ (2) are direct, it remains to prove (1) $\Rightarrow$ (4).
Now let $(X, T)$ be a weakly mixing system, and then $\Tran(X^2, T^{(2)})$ is a dense $G_\delta$ subset of $X^2$. It suffices to show $\omega_{\mathcal{N}_T}(x)=X$ for each point $(x, y) \in \Tran(X^2, T^{(2)})$.
Let $G$, $U$, $V$ be opene sets in $X$. Obviously, $y \in \Tran(X, T)$, and hence there is $m \in \mathbb{Z}_+$ such
that $T^m y \in U$. Since $(x, y)$ is transitive, there
exists $n \in N_T(x, G) \cap N_T(y, T^{-m}V) \neq \varnothing$. Then $T^nx \in G$ and $T^n U \cap V \neq \varnothing$,
because $T^{n+m}y \in V$ and $T^my \in U$. Hence for any opene sets $G$, $U$, $V$, $N_T(x, G) \cap N_T(U, V) \neq \varnothing$.
That is, $\omega_{\mathcal{N}_T}(x)=X$.
\end{proof}

Applying Corollary \ref{201509031402} (2) to Proposition \ref{0203}, we have following direct corollaries:

\begin{cor} \label{20160110}
A dynamical system $(X, T)$ is minimal and weakly mixing if and only if $\omega_{\mathcal{N}_T}(x)=X$ for each $x\in X$.
\end{cor}

\begin{cor} \label{201509061135}
Let $(X, T)$ be an $M$-system. Assume $\omega_{\mathcal{N}_T}(x)\neq \varnothing$ for any minimal point $x\in X$.
Then $(X, T)$ is weakly mixing. In particular, each transitive compact $M$-system is weakly mixing.
\end{cor}

Remark that Downarowicz and Ye constructed a ToP system with positive entropy which is not totally transitive (and hence not weakly mixing) \cite[Theorem 2]{Downarowicz-Ye}, recall that a system $(X, T)$ is a \emph{ToP system} if every point is either transitive or periodic. Thus by Corollary \ref{201509061135} the constructed system is not transitive compact.

Note that if the system $(X,T)$ is
topologically mixing then it is easy to show from the definitions that
$\omega_{\mathcal{N}_T}(x)$ coincides with $\omega_{T}(x)$ for any point $x\in  X$.
In the following let us show that it may happen $\omega_{\mathcal{N}_T}(x)\subsetneq
\omega_T(x)$  for a weakly mixing system $(X,T)$ and some point $x\in X$. Remark that in Proposition \ref{201509042231} we have proved the invariance of the $\omega_{\mathcal{N}_T}$-limit sets.

\begin{thm} \label{newprop}
There is a weakly mixing system $(X, T)$ and a point $x \in X$ such that $\omega_{\mathcal{N}_T}(x)\subsetneq \omega_T(x).$
\end{thm}

\begin{proof} Let $\Sigma=\{0,1\}^{\mathbb{Z}_+}$ and  $\sigma :\Sigma \to \Sigma$ be the full (one-side) shift.
Recall that the base for the open sets in $\Sigma$ is given by the collection of all cylinder sets
$C[c_0c_1c_2 \dots c_m]=\{x \in \Sigma:~ x_i=c_i \text{ for } i\leq m \}.$

Let $P$ be a subset of $\mathbb{N}$ and we define $\Lambda_P=\{x \in \Sigma: x_i=x_j=1,~i \neq j \Rightarrow |i-j| \in P\}.$
The set $\Lambda_P$ is a closed $\sigma$-invariant subset of $\Sigma$, i.e., $\Lambda_P$ is a subshift (in the case of $\Lambda_P\neq \varnothing$).
So we can consider the dynamical system
$ (X, T):= (\Lambda_P, \sigma|_{\Lambda_P})$. In \cite{LaZa} (see also \cite{JiSn2003}) the authors proved that if $P$ is thick,
 then $(X, T)$ is weak mixing, and if $\mathbb{N} \setminus P$ is infinite, then $(X, T)$ is
 not topologically mixing.

Let $P:=\{10^n+s: n \in \mathbb{N}, 1 \leq s \leq n \}.$ Obviously $P$ is thick. $(X, T)$ is a weakly mixing system,
but not topologically mixing. So, there is a chance for it to have a point with the desired property.
We are going to define points $x, y \in X$ such that there is a neighbourhood $U_y$ of $y$ and opene $U$, $V$
with $N_T(x, U_y) \cap N_T(U, V) = \varnothing$.

Let $y=10^{10}10^\infty$. We have $y \in X$ since the only positions of one's in this sequence are $0$ and $11$ and $11 \in P$.
Denote $U_y=C[10^{10}1]$.

We define $x$ as a limit of a sequence of the starting blocks $(W_n)_{n=1}^\infty$. We do it by induction. On every step in
order to define the next block we add to the previous one some amount of one's and zero's. Let $W_0=W=10^{10}1$ and suppose that
we have already defined $W_n$. Let $b_n$ be the position of the last digit in $W_n$
(the length of $W_n$ is $b_n+1$). To obtain $W_{n+1}$ we define the digits on the positions $s$ for
$b_n < s \leq 10^{b_n+12}+b_n+12$. We set $1$ on the following positions:
$10^{b_n+12}+b_n+12$ and $10^{b_n+12}+b_n+1$. For the other $s$ from this interval let $x_s=0$.
We see that $W_{n+1}=W_n0^{a_n}W$ for $a_n=10^{b_n+12}$ and then $x=W0^{a_1}W0^{a_2}W\dots$ In order to prove that $x \in X$
we need to check if for every $n \geq 0$, $W_{n+1}$ satisfies the condition of the
space $X$. Under the
assumption that $W_n$ satisfies this condition we need to check only the elements which we have added. Among these elements
we had only two one's. The difference between the numbers of their positions is $11$. We have that $10^{b_n+12}+s \in P$,
where $s \in [1, b_n+12]$. Then, for any $0 \leq i \leq b_n$ with $x_i=1$ we have $10^{b_n+12}+b_n+12-i \in P$ and
$10^{b_n+12}+b_n+1-i \in P$. Hence $x$ which we have constructed as a limit of the blocks $W_n$ belongs to $X$. Finally, it is easy
to see that $a_n=10^{b_n+12}>2b_{n} \geq 2a_{n-1}$ so $a_n$
tends to infinity. Moreover, $y \in \omega_T(x)$ follows from
$N_T(x, U_y)=\{10^{b_n+12}+b_n+1 | n \in \mathbb{N}\}$, because the numbers which belong to this set are exactly the
positions where every block $W$ starts.

Now we need to define $U$ and $V$. Let $U=U_y=C[10^{10}1]$ and $V=C[010^{10}1]$. These two open sets have a nonempty intersection
with $X$, and so $N_T(U, V)$ is a thick set. We have that $s+12 \in N_T(U, V)$ (implying $s\in \mathbb{N}$) if and only if there
exists a point $z \in X$ which begins by $10^{10}1A(s)010^{10}1$ where the length of $A(s)$ is equal to $s$.
Thus we have $s+13, s+24, s+2 \in P$. Then $10^m+1 \leq s+2, s+13, s+24 \leq 10^m+m$ for some $m\in \mathbb{N}$. So $10^m+11 \leq s+12
\leq 10^m+m-12$. So in other words for every $k \in N_T(U, V)$ there is $m\in \mathbb{N}$ such that $10^m+11 \leq k \leq 10^m+m-12$.
But now it is easy to see that $k \neq 10^{b_n+12}+(b_n+12)-11 (= 10^{b_n+12}+ b_n+ 1)$, so $k$ can not be an element of $N_T(x, U_y)$.
\end{proof}


Though a transitive compact system is weakly mixing if we add some weak sufficient conditions over the system, there exist many transitive compact, but not weakly mixing systems. Before proceeding, first we need the following property of the proximal systems. Recall that a system is
proximal if and only if it contains the unique fixed point, which is the only minimal point of the system \cite{AK}.

\begin{lem} \label{0205}
Let $(X, T)$ be a proximal dynamical system, $p\in X$ its fixed point with a neighborhood $U_p$ and $x\in X$.
Then the set $N_T(x, U_p)$ is thickly syndetic.
\end{lem}

\begin{proof}
First we are going to prove that $N_T(x, U_p)$ is a syndetic set. As $p$ is the unique minimal point of the system,
for any point $z \in X$ there is $m(z)\in \mathbb{Z}_+$ with
$T^{m(z)}z \in U_p$, and then the family of open sets $\{T^{-i}U_p: i \in \mathbb{Z}_+\}$ forms an open cover
of $X$. By the compactness of the space $X$, there exists an integer $N\in \mathbb{N}$ such that
$\{U_p, T^{- 1} U_p, \dots , T^{-N}U_p\}$ forms a cover of $X$.
Then, for every $m\in \mathbb{Z}_+$, $T^{m+ i}x\in U_{p}$ for some $0\le i\le N$, which shows that $N_T(x, U_p)$ is a
syndetic set with gaps bounded by the integer $N$.

Now for each $n\in \mathbb{N}$ we define $U_{p,n}=\bigcap_{i=0}^n T^{-i}U_{p}$, which is a neighborhood
of $p$ as $U_p$ is a neighborhood of the fixed point $p$. Observe $\{m, m+1, \dots, m+n\} \subset N_T(x, U_{p})$
for each $m \in N_T(x, U_{p,n})$, where $N_T(x, U_{p,n})$ is a syndetic set by the above argument. Then $N_T(x, U_p)$
is thickly syndetic by the arbitrariness of $n$.
\end{proof}

We also need the following

\begin{lem} \label{201509062214}
Let $(X, T)$ be a minimal weakly mixing system, $x\in X$ with a neighborhood $U_x$ and $n, m\in \mathbb{N}$. Then the set $\{k\in \mathbb{Z}_+: T^{k n+ m} x\in U_x\}$ is syndetic.
\end{lem}

\begin{proof}
Observe that each weakly mixing system is totally transitive, in particular, $(X, T^n)$ is transitive. Note that each minimal point of $(X, T)$ is also a minimal point of $(X, T^n)$ by \cite[Corollary 3.2]{Ye1992}, in particular, any point $x\in X$ is a minimal point of $(X, T^n)$. Then we obtain that $(X, T^n)$ is a minimal system.

Now assume the contrary that the set $\{k\in \mathbb{Z}_+: T^{k n+ m} x\in U_x\}$ is not syndetic, and then there exists an increasing sequence $k_1< k_2< \dots$ such that $\{T^{k_i n+ m} x, T^{(k_i+ 1) n+ m} x, \dots, T^{(k_i+ i) n+ m} x\}\subset U_x^c$ for each $i\in \mathbb{N}$. By choosing a subsequence if necessary we may assume that the sequence $T^{k_i n+ m} x$ tends to $x_0\in X$ as $i$ tends to infinity. And then $\orb_{T^n} (x_0)= \{x_0, T^n x_0, T^{2 n} x_0, \dots\}\subset \overline{U_x^c}\subsetneq X$ by the construction, a contradiction to the minimality of the system $(X, T^n)$. That is, $\{k\in \mathbb{Z}_+: T^{k n+ m} x\in U_x\}$ is a syndetic set.
\end{proof}

Now we can show the existence of non totally transitive, transitive compact systems. Recall that each weakly mixing system is totally transitive.

\begin{thm} \label{0206}
There are non totally transitive, transitive compact systems.
\end{thm}

\begin{proof} We will show the existence of such systems in proximal case, and then by modifying the construction to show the existence in non-proximal case.


1. \emph{Proximal case}. Let $(Y,F)$ be a nontrivial proximal, topologically mixing system (for the existence of such a system see some example \cite{HeZ2002}). Let $q\in Y$ be its (unique) fixed point and take a copy of it:
the system $(Y_c, F_c)$ with its (unique) fixed point $q_c\in Y_c$. Suppose that $Y$ and $Y_c$ are disjoint and consider
the wedge sum $X:= Y\vee Y_c$, i.e., the quotient space of the disjoint union of $Y$ and $Y_c$ by identifying $q$ and $q_c$, and then both topological spaces $(Y,q)$ and $(Y_c, q_c)$ with base points look like subspaces of the wedge sum with the subspace
topology. Let us define a self-map $T$ over $X$ as follows:
\[T:x \mapsto
  \begin{cases}
    F_cy_c &\mbox {if }
x =y \in Y\\
    Fy &\mbox{if } x = y_c \in Y_c
  \end{cases}
\]
for any point $x \in X$. By the construction it is not hard to show that the map $T: X\to X$ is
a continuous surjection; and the system $(X, T)$ is not totally
transitive (observing $T^2 Y\subset Y$ and $T^2 Y_c\subset Y_c$) and proximal with the unique minimal (in fact fixed) point $p:= q\sim q_c$, moreover, it is a transitive system, in fact, for any opene subsets $U,V$ of $X$ we may find $k\in \mathbb{Z}_+$
with $N_T(U, V)\supset \{k, k+2, \dots\}$ (observing that the systems $(Y,F)$ and its copy $(Y_c,F_c)$ are both topologically mixing).

Now let us show that the system $(X,T)$ is transitive compact by proving $p\in \omega_{\mathcal{N}_T}(x)$ for each point $x\in X$. Let $G_p$ be a neighborhood of the fixed point $p\in X$ and $U,V$ be opene
subsets of $X$.
As $N_T(x, G_p)$ is thickly syndetic by Lemma \ref{0205} and $N_T(U, V)\supset \{k, k+2, \dots\}$ for some $k\in \mathbb{Z}_+$, we have $N_T(x, G_p) \cap N_T(U, V)\ne \varnothing$.

In fact we can obtain more by proving $\omega_{\mathcal{N}_T}(x)= \{p\}$ for each $x\in X$ (note that since $(X,T)$ is topologically transitive,
there exists a dense $G_\delta$ subset $X_0$ of $X$ with $\omega_T (x)= X$ for each $x\in X_0$, in particular, $\omega_{\mathcal{N}_T}(x)\subsetneq \omega_T (x)$ for each $x\in X_0$). It is easy to check it for the case of $x= p$.
  Now assume that $x\in Y\setminus Y_c$.
Let $z\in Y$ be a point different from $p$. Take a neighborhood $G_z\subset Y\setminus Y_c$. It is easy to check that $N_T (x, G_z)$ consists of only even numbers. Now in $Y$ we take an opene subset $U$ of $X$ and in $Y_c$ we take an opene subset $V$ of $X$. It is also easy to check that $N_T (U, V)$ consists of only odd numbers, in particular, $N_T (x, G_z)\cap N_T (U, V)= \varnothing$, and then $z\notin \omega_{\mathcal{N}_T}(x)$. We can prove similarly that $\omega_{\mathcal{N}_T}(x)\cap Y_c\setminus Y= \varnothing$. Summing up, we obtain $\omega_{\mathcal{N}_T}(x)= \{p\}$. The case of $x\in Y_c\setminus Y$ can be done similarly.


2. \emph{Non-proximal case}.
Now we are going to construct a non totally transitive, non-proximal, transitive compact system by modifying the above construction.

Let $(X, T)$ be the non totally transitive, proximal, transitive compact system constructed as above.
We take a nontrivial minimal topologically mixing system
$(Z,R)$ (in particular, it is not proximal).
Then the product system $(X\times Z, T\times R)$ is a non totally transitive, non-proximal system.

Now let us show that the system $(X\times Z, T\times R)$ is transitive compact by proving $\omega_{\mathcal{N}_{T\times R}}(x, z)= \{p\}\times Z$ for each point $x\in X$ and $z\in Z$. Let $G_p$ be a neighborhood of the point $p\in X$, $G_z$ be a neighborhood of the point $z\in Z$ and $U,V$ be opene
subsets of $X\times Z$.
It is easy to obtain from the construction that $N_{T\times R} (U, V)\supset \{k, k+ 2, \dots\}$ for some $k\in \mathbb{Z}_+$. As $N_T(x, G_p)$ is thickly syndetic by Lemma \ref{0205}, and $N_R (z, G_z)\cap \{k, k+2, \dots\}$ is a syndetic set by applying Lemma \ref{201509062214} to the minimal topologically mixing (and hence weakly mixing) system $(Z, R)$, we have
$$N_{T\times R} ((x, z), G_p\times G_z)\cap N_{T\times R} (U, V)= N_T (x, G_p)\cap N_R (z, G_z)\cap N_{T\times R} (U, V)\neq \varnothing.$$
 This shows $(p, z)\in \omega_{\mathcal{N}_{T\times R}}(x, z)$, and then $\omega_{\mathcal{N}_{T\times R}}(x, z)\supset \{p\}\times Z$ by applying Lemma \ref{0200} (recall that the system $(Z, R)$ is minimal). Moreover, observing $\omega_{\mathcal{N}_T} (x)= \{p\}$ one has $\omega_{\mathcal{N}_{T\times R}}(x, z)\subset \{p\}\times Z$, and hence finally $\omega_{\mathcal{N}_{T\times R}}(x, z)= \{p\}\times Z$.
\end{proof}


\section{Implications among transitive compactness\\
 and various stronger forms of sensitivity}

In this section we study the relationships among transitive compactness and other various stronger forms of sensitivity.
Observe that multi-sensitivity and thick sensitivity are equivalent for transitive systems \cite[Proposition 3.2]{HKZ}, we will show that
 multi-sensitivity is equivalent to both thick sensitivity and thickly syndetic sensitivity for M-systems, any minimal spatio-temporally chaotic system is thickly sensitive, and that any transitive compact system is Li-Yorke sensitive and furthermore multi-sensitive if it is not proximal.


Any cofinitely sensitive system is clearly thickly syndetically sensitive,
any thickly syndetically sensitive system is multi-sensitive because the intersection of finitely many thickly syndetically sets is also thickly syndetic, and any multi-sensitive system is thickly sensitive \cite{Subrahmonian2007} and \cite[Proposition 3.2]{HKZ}.

The following Figure 1 presents a comparison between stronger forms of sensitivity for general topological
dynamical systems.

\vskip 10pt

\begin{center}
\begin{tikzpicture}[
level 1/.style={sibling distance=40mm},
  edge from parent/.style={->, double},
  >=latex]


\begin{scope}[every node/.style={level 3}]
\node [xshift=-0pt,yshift=-20pt] (c11) {Thickly \\ syndetical sensitivity};
\node [xshift=-130pt,yshift=-20pt] (c1) {Cofinite sensitivity};
\node [xshift=130pt,yshift=-20pt] (c21) {Multi-sensitivity};
\node [below of = c11, xshift=0pt,yshift=-20pt] (c41) {Sensitivity};
\node [below of = c21, xshift=0pt, yshift=-20pt] (c31) {Thick sensitivity};
\end{scope}
\draw[->, double] (c1) -- (c11);
\draw[->, double] (c11) -- (c21);
\draw[->, double] (c31) -- (c41);
\draw[->, double] (c21) -- (c31);
\end{tikzpicture}

\vskip 10pt

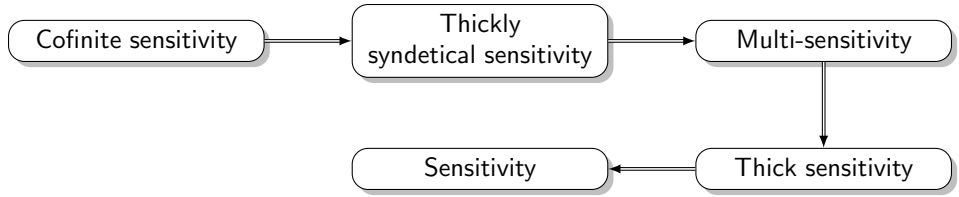
Figure 1. General case.
\end{center}


In the following we will discuss the relationship between these various sensitivity when we require some strong transitivity over the systems.

\begin{prop} \label{1407082207}
If $(X, T)$ is a thickly sensitive M-system, then $(X, T)$ is thickly syndetically sensitive.
\end{prop}
\begin{proof}
Recall that if $(X, T)$ is an M-system, then $(X^k, T^{(k)})$, the product system of $k$ copies of $(X, T)$,
 contains a dense set of minimal points for any $k\in \mathbb{N}$. In fact we can say more. Let $x\in \Tran (X, T)$ and $U_1, \dots, U_k$ be opene subsets in $X$. There are
  $n_1, \dots, n_k\in \mathbb{N}$ such that $T^{n_1} x\in U_1, \dots, T^{n_k} x\in U_k$. Since $(X, T)$ is an M-system, there is a minimal
  point $x_0\in X$  sufficiently close to $x$ such that $T^{n_1} x_0\in U_1, \dots, T^{n_k} x_0\in U_k$.  Thus $U_1\times \dots\times U_k$
  contains the minimal point $(T^{n_1} x_0, \dots, T^{n_k} x_0)$ of $(X^k, T^{(k)})$.

Let $(X, T)$, which is thickly sensitive, have a sensitivity constant $\delta> 0$. Let $U$ be an opene subset in $X$ and $k\in \mathbb{N}$.
 Since  $S_T (U, \delta)$ is a thick set,
 $$\bigcap_{i= 0}^k S_T (T^{- i} U, \delta)\supset \{n\ge k: \{n- k,
 \dots, n- 1, n\}\subset S_T (U, \delta)\}.$$
Therefore there are a positive integer $n_0\in \bigcap_{i= 0}^k S_T (T^{- i} U, \delta)$, $n_0\ge k$, and
  $x_i, y_i\in T^{- i} U$ with $d (T^{n_0} x_i, T^{n_0} y_i)> \delta$,~ $i= 0, 1, \dots, k$.
  Moreover, we can choose opene subsets $U_i, V_i\subset T^{- i} U$ such that $d (T^{n_0} x_i', T^{n_0} y_i')> \delta$ for all
   $x_i'\in U_i$ and $y_i'\in V_i$,~ $i= 0, 1, \dots, k$.
   Again, since  $(X, T)$ is an M-system, the system $(X^{2 k+ 2}, T^{(2 k+ 2)})$ contains a
   dense set of minimal points, and there is a minimal point $(z_0, z_0', z_1, z_1', \dots, z_k, z_k')\in U_0\times
    V_0\times U_1\times V_1\times \dots\times U_k\times V_k$. Obviously,
$$\mathcal{S}= \bigcap_{i= 0}^k (N_T (z_i, U_i)\cap N_T (z_i', V_i))$$
is a syndetic set. From the construction we get that  $S_T (U, \delta)\supset \{m+ n_0 - i: m\in \mathcal{S}, i= 0, 1, \dots, k\}$,
 which is a thickly syndetic set.
\end{proof}

In particular, we have:

\begin{thm} \label{1407090448}
Thickly syndetical sensitivity, multi-sensitivity and thick sensitivity are all equivalent properties for an M-system.
\end{thm}

The following Figure 2 presents a comparison between stronger forms of sensitivity for $M$-systems. Remark that any non-minimal M-system is thickly syndetically sensitive \cite[Theorem 8]{liuheng}.

\vskip 10pt

\begin{center}
\begin{tikzpicture}
[
  level 1/.style={sibling distance=70mm},
  edge from parent/.style={->, double},
  >=latex]

\begin{scope}[every node/.style={level 2}]
 \node[xshift=0pt,yshift=-5pt] (c1) {Cofinite sensitivity};

\begin{scope}[every node/.style={level 3}]
\node [below of = c1, xshift=0pt,yshift=-20pt] (c11) {Thickly \\ syndetical sensitivity};
\node [below of = c1, xshift=-125pt,yshift=-20pt] (c21) {Multi-sensitivity};
\node [below of = c1, xshift=125pt,yshift=-20pt] (c31) {Thick sensitivity};
\node [below of = c11, xshift=0pt,yshift=-20pt] (c41) {Sensitivity};

\end{scope}
\end{scope}
\draw[->, double] (c1) -- (c11);
\draw[<->, double] (c11) -- (c21);
\draw[->, double] (c11) -- (c41);
\draw[<->, double] (c11) -- (c31);

\end{tikzpicture}
\vskip 10pt

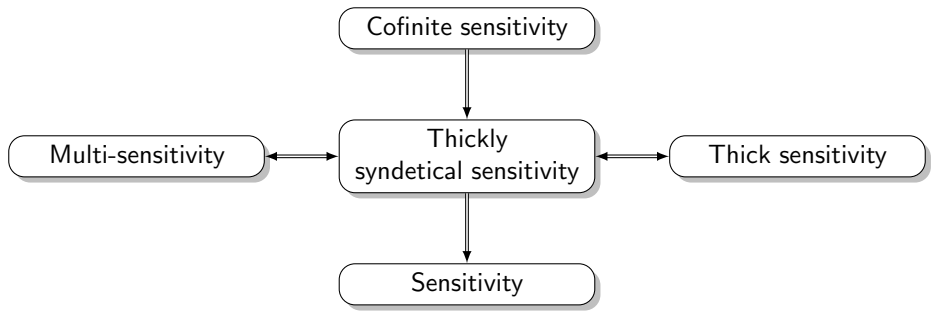
Figure 2. M-systems.

\vskip 10pt
\end{center}

The dichotomy theorem for minimal thickly sensitive systems, \cite[Theorem 3.1]{HKZ}, states that a minimal system is not thickly sensitive if and only if it is an almost one-to-one extension of its maximal equicontinuous factor.
Because any transitive non-sensitive system is uniformly rigid \cite[Lemma 1.2]{GlasnerWeiss1993} and hence has zero topological entropy \cite[Proposition 6.3]{GlasnerMaon1989}, then: any Toeplitz flow with positive topological entropy is a minimal sensitive invertible system which is not thickly sensitive (see
\cite{Downarowicz2005}
 for a detailed construction of such a system).


When discussing minimal systems, we have the following result:

\begin{prop} \label{201507052153}
Let $(X, T)$ be a minimal system. If $(X, T)$ is spatio-temporally chaotic then it is thickly sensitive.
\end{prop}
\begin{proof}
Let $\pi_\text{eq}: (X, T)\rightarrow (X_\text{eq}, T_\text{eq})$ be the associated factor map from $(X, T)$ to its maximal equicontinuous factor.
Applying the dichotomy theorem \cite[Theorem 3.1]{HKZ} to the minimal system $(X, T)$ we obtain that $(X, T)$
is not thickly sensitive if and only if $\pi_\text{eq}: X\rightarrow X_\text{eq}$ is almost one-to-one. Assume the contrary that $(X, T)$ is not thickly sensitive. Then we can take $x\in X$ such that $\pi^{- 1} (\pi x)$ is a singleton, from which it is easy to see that $(x, z)$ can not be a Li-Yorke pair for any $z\in X\setminus \{x\}$. Therefore $(X, T)$ is not spatio-temporally chaotic, a contradiction.
\end{proof}

Remark that in \cite{YeYu} Ye and Yu introduced and discussed block sensitivity and strong sensitivity for several families. Applying the dichotomy theorems obtained by them in \cite{YeYu} to the same proof of Proposition \ref{201507052153} one has that any minimal spatio-temporally chaotic system is strongly IP-sensitive.


Note that the converse of Proposition \ref{201507052153} does not hold. Precisely, there are invertible minimal cofinitely sensitive systems which are distal (and hence not spatio-temporally chaotic): many standard examples provided in \cite{Furstenberg1981} are such systems.
For example, we consider a dynamical system $(X, T)$ given by
$X= \mathbb{R}^2/\mathbb{Z}^2\ \text{ and }\ T: (x, y)\mapsto (x+ \alpha, x+ y)$ for $\alpha\notin \mathbb{Q}$.
It is well known that $(X, T)$ is an invertible minimal distal system (see \cite[Chapter 1]{Furstenberg1981}).
Now for any opene $U\subset X$ take $x_0, y_0\in \mathbb{R}/\mathbb{Z}$ and $\delta> 0$
with $(x_0- \delta, x_0+ \delta)\times \{y_0\}\subset U$. Since
 $T^n (x, y)= \left(x+ n \alpha, n x+ \dfrac{n (n- 1)}{2} \alpha+ y\right)$
 for any point $(x, y)\in X$ and any $n\in \mathbb{N}$,
 the diameter of $T^n U$ is at least the length of the circle $\mathbb{R}/ \mathbb{Z}$ when $n\in \mathbb{N}$ is
 large enough. That is, the system $(X, T)$ is cofinitely sensitive.


Also note that, as shown by the system constructed in \cite[Example 3.7]{HKZ}, the assumption of minimality in Proposition \ref{201507052153} can not be relaxed.
In \cite[Example 3.7]{HKZ} a transitive non-minimal sensitive system is constructed such that 1) it is not thickly sensitive; and
2)
it contains a fixed point as its unique minimal set, and hence the system is Li-Yorke sensitive by \cite[Corollary 3.7]{AK}.



It is shown that each minimal weakly mixing system is thickly syndetically sensitive \cite[Theorem 7]{liuheng},
which can extended a little bit as follows.
Observe that the proofs of \cite[Lemma 3 and Theorem 8]{liuheng} show that any topologically ergodic system containing two different minimal subsets is
 thickly syndetically sensitive, and
remark that we have assumed that all dynamical systems interested are nontrivial.

\begin{prop} \label{1407082217}
If a system $(X,T)$ is thickly syndetically transitive, then $(X,T)$ is thickly syndetically sensitive.
\end{prop}
\begin{proof}
We take opene $V_1, V_2\subset X$ with $\delta= \dist (V_1, V_2)> 0$.
Now let $U\subset X$ be an opene subset. By the assumption, both $N_T (U, V_1)$ and $N_T (U, V_2)$ are thickly syndetic, and hence
$S_T (U, \delta)\supset N_T (U, V_1)\cap N_T (U, V_2)$ is also thickly syndetic.
\end{proof}


 The following results show sensitivity in a transitive compact system.

\begin{thm} \label{0208}
Any transitive compact system $(X, T)$ is Li-Yorke sensitive.
\end{thm}

\begin{proof}
As for every $x\in X$ the proximal cell $\Prox _T (x)\subset X$ is a dense subset by Proposition \ref{0207}, then by \cite[Theorem 3.6]{AK} it suffices to show that $(X,T)$ is sensitive.

As we have assumed that we are only interested in nontrivial dynamical systems. Let $\dfrac{\diam(X)}{4} > \epsilon > 0$ and define $\delta =\dfrac{\diam(X)}{2}-2\epsilon> 0$. We will show that $\delta$ is a sensitive constant of $(X, T)$.
Let $x\in X$ with a neighborhood $U_x$. Let $y \in \omega_{\mathcal{N}_T}(x)$ and take a point $x_1\in X$ with $d(x_1, y)\ge \dfrac{\diam(X)}{2}$. Now consider opene subsets $B_y\ni y$ and $B_{x_1}\ni x_1$ such that their diameters is strictly smaller than $\epsilon$. As $y \in \omega_{\mathcal{N}_T}(x)$,
we may find an integer $m\in \mathbb{Z}_+$ and a point $z\in U_x$ such that $T^mx \in B_y$ and  $T^m (z)\in B_{x_1}$, and then $d (T^m z, T^m x)\geq d(y, x_1)- d(T^m x, y)-d(T^m z,x_1)> \delta.$
 \end{proof}

  In particular, any transitive compact system is sensitive and hence $\mathbb{L}_{r}> 0$ and $\overline{\mathbb{L}}_{r}>0$. In fact, this point can be strengthened as follows, which was essentially proved in \cite{Kolyada2013}, although there it was done for weakly mixing systems.

\begin{prop} \label{0209}
Let $(X,T)$ be a transitive compact system. Then $\mathbb{L}_{r}=\overline{\mathbb{L}}_{r}> 0$.
\end{prop}

Moreover, the proof of Theorem \ref{0208} shows that it takes place $\mathbb{L}_d\ge \dfrac{\diam(X)}{2}$ for any transitive compact system $(X, T)$, while this estimate is exact as shown by the proximal, non weakly mixing, transitive compact system $(X, T)$ constructed in Theorem \ref{0206}: it is not hard to show $\mathbb{L}_d = \dfrac{\diam(X)}{2}$ for the system $(X, T)$.


Recall that any transitive compact system is transitive, and thick sensitivity is equivalent to multi-sensitivity for transitive systems \cite[Proposition 3.2]{HKZ}.

 \begin{thm} \label{0206a}
Any non-proximal, transitive compact system is thickly sensitive.
\end{thm}

\begin{proof}
Let $(X,T)$ be a non-proximal, transitive compact system (and hence it is transitive), and let $x \in \Tran (X, T)$.
As $(X,T)$ is transitive compact, $\omega_{\mathcal{N}_T}(x)$ is a nonempty, positively $T$-invariant, closed subset of $X$ by Lemma \ref{0200}, and then we may take a minimal point $y \in  \omega_{\mathcal{N}_T}(x)$.
 Since $(X,T)$ is non-proximal, there exists a point $z\in X$ such that $\liminf_{n \rightarrow \infty} d(T^n y, T^n z)> \delta$ for some $d (y, z)> \delta> 0$: if $y$ is a fixed point, then applying \cite{AK} to non-proximal $(X, T)$ we obtain another minimal point $z$ different from $y$; if $y$ is not a fixed point then we define $z= T y$.

Now let $U$ be any opene subset of $X$.
As $x \in \Tran (X, T)$, there exist an opene subset $U_0$ containing $x$ and an integer $n_{- 1}\in \mathbb{Z}_+$ with $T^{n_{- 1}} U_0\subset U$.
Now let $U_y^0$, $U_z^0$ be neighborhoods of $y, z$ respectively such that the distance between $U_y^0$ and $U_z^0$ is strictly larger than $\delta$.
As $(X,T)$ is transitive compact, by Lemma \ref{02000} we may take $n_0> n_{- 1}$ such that $T^{n_0}x \in U_y^0$ and
$T^{n_0}(U_0) \cap U_z^0 \neq \varnothing$, and then let $U_1$ be an opene set of $X$ with $U_1 \subset \overline{U_1} \subset U_0 \cap T^{-n_0}(U_z^0)$.
Now assume that by induction we have defined $U_k$, $k \geq 1$. Let
$U_y^k$, $U_z^k$ be neighborhoods of points $y, z$ respectively such that the distance between $T^{i}(U_y^k)$ and $T^{i}(U_z^k)$ is strictly larger than $\delta$ for all $0 \leq i \leq k$.
Applying again Lemma \ref{02000} to transitive compact $(X,T)$ we may take $n_k>n_{k-1}$ such that $T^{n_k}x \in U_y^k$ and $T^{n_k}(U_k) \cap U_z^k\neq \varnothing$, and then let $U_{k+1}$ be an opene set of $X$ with $U_{k+1} \subset \overline{U_{k+1}} \subset U_k \cap T^{-n_k}(U_z^k)$. Obviously,
$\bigcap_{k=1}^{\infty} \overline{U_k}=\bigcap_{k=1}^{\infty} U_k$ is nonempty and from it we take a point $x_0$. From the above construction, it is easy to obtain $S_T (U, \delta)\supset \{n- n_{- 1}: d(T^n x, T^n x_0)>\delta\}$ and
$$\{ n\in \mathbb{Z}_+: d(T^n x, T^n x_0)>\delta \}\supset \bigcup_{k\in \mathbb{N}} \{n_k, n_k+ 1, \dots, n_k+ k\},$$
in particular, $S_T (U, \delta)$
 is a thick set. Thus, $(X, T)$ is thickly sensitive.
\end{proof}

The following Figure 3 presents a comparison between stronger forms of sensitivity for transitive systems.
Remark that it is direct to see from the definition that each weakly mixing system is multi-sensitive.

\vskip 10pt

\begin{center}
\begin{tikzpicture}
[
  level 1/.style={sibling distance=40mm},
  edge from parent/.style={->, double},
  >=latex]

\node[root] (c0) {Topologically mixing}
  child {node[level 2, xshift=0pt,yshift=0pt] (c2) {Weak mixing}};
\begin{scope}[every node/.style={level 3}]
\node [below of = c0, xshift=-118pt,yshift=-60pt] (b11) {Thick sensitivity};
\node [below of = c0, xshift=0pt,yshift=-60pt] (d21) {Transitive compactness};
\node [below of = c0, xshift=118pt,yshift=-60pt] (c012) {Dense \\ proximal cells};
\node [below of = c0, xshift=118pt,yshift=-110pt] (c21) {$\mathbb{L}_{r}=\overline{\mathbb{L}}_{r} > 0$};
\node [below of = c0, xshift=-118pt,yshift=-110pt] (c011) {Multi-sensitivity};

\node [below of = c0, xshift=0pt,yshift=-110pt] (b21) {Li-Yorke sensitivity};
\node [below of = c0, xshift=0pt,yshift=-150pt] (c211) {Sensitivity};
\end{scope}
\draw[->, double] (c0) -- (c2);
\draw[->, double] (c2) -- (d21);
\draw[->, double] (d21) -- (c012);
\draw[->, double] (c2) -- (b11);
\draw[->, double] (d21) -- (c21);
\draw[->, double] (b21) -- (c211);
\draw[->, double] (c011) -- (c211);
\draw[->, double] (d21) -- (b21);
\draw[<->, double] (b11) -- (c011);
\end{tikzpicture}
\vskip 10pt

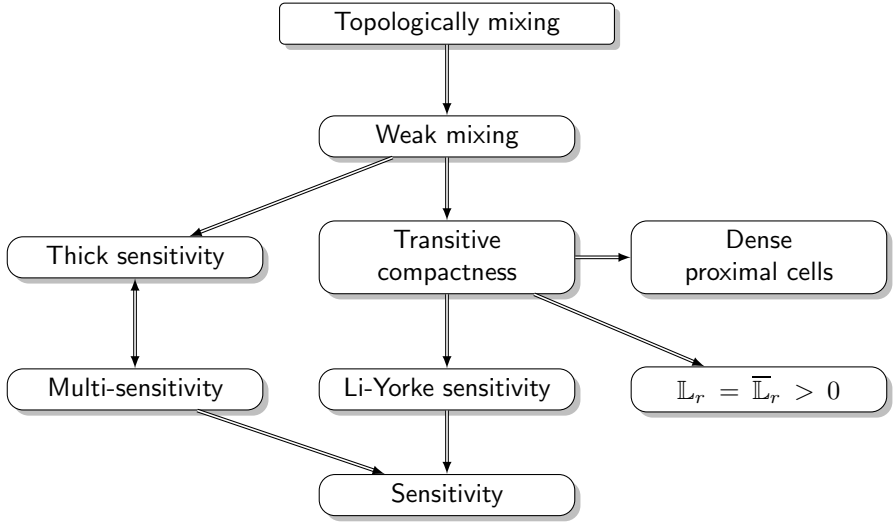
Figure 3. Topologically transitive systems.

\vskip 10pt
\end{center}





Remark that each transitive compact system is Li-Yorke sensitive, and that there are non-minimal M-systems which are not
transitive compact (for example the ToP system constructed by Downarowicz and Ye in \cite[Theorem 2]{Downarowicz-Ye}).

\section{Quantitative analysis for multi-sensitivity}

In this section, to measure the multi-sensitivity of a dynamical system, we are interested in relationships between the introduced $\mathbb{L}_{m,r},
\overline{\mathbb{L}}_{m,r}, \mathbb{L}_{m,d}$ and $\overline{\mathbb{L}}_{m, d}$.


It is easy to see that $(X, T)$ is multi-sensitive if and only if $\mathbb{L}_{m,d}> 0$, and
\begin{equation} \label{1407052335}
\mathbb{L}_{m,d}\ge \mathbb{L}_{m,r}\ge \overline{\mathbb{L}}_{m,r} \text{ and }
\mathbb{L}_{m,d}\ge \overline{\mathbb{L}}_{m,d} \ge \overline{\mathbb{L}}_{m,r}.
\end{equation}
Moreover, we have the following

\begin{lem} \label{1407052201}
$\mathbb{L}_{m,d}\le 2 \overline{\mathbb{L}}_{m,r}$.
\end{lem}
\begin{proof}
We only need to consider the case $\mathbb{L}_{m,d}> 0$ following \cite[Proposition 2.1]{Kolyada2013}.

Let $\epsilon> 0$ be small enough with $\mathbb{L}_{m,d}> 2 \epsilon$. Now consider a collection of points $x_i$ ($i= 1, \dots, k$) with neighborhoods $U_i\ni x_i$. Set $V_{0, 1}= U_1, \dots, V_{0, k}= U_k$. Take
$$n_0\in \bigcap_{i= 1}^k S_T (V_{0, i}, \mathbb{L}_{m,d}- \frac{\epsilon}{2}).$$
Remark that for any $\delta> 0$ if $d (y_1, y_2)> 2 \delta$ then for all $x\in X$ either $d (y_1, x)> \delta$ or $d (y_2, x)> \delta$. And then there exist $y_{0, 1}\in V_{0, 1}, \dots, y_{0, k}\in V_{0, k}$ such that
\begin{equation} \label{1407060224}
\min_{1\le i\le k} d (T^{n_0} x_i, T^{n_0} y_{0, i})> \frac{\mathbb{L}_{m,d}- \epsilon}{2}.
\end{equation}
 Moreover, we can choose open neighborhoods $V_{1, 1}$ of $y_{0, 1}$ (with $\overline{V_{1, 1}}\subset V_{0, 1}$),
 $\dots$, $V_{1, k}$ of $y_{0, k}$ (with $\overline{V_{1, k}}\subset V_{0, k}$) such that
\begin{equation} \label{1407060225}
\max_{0\le n\le n_0} \max_{1\le i\le k} \diam (T^n V_{1, i})\le \frac{\epsilon}{2}.
\end{equation}
Again take
$$n_1\in \bigcap_{i= 1}^k S_T (V_{1, i}, \mathbb{L}_{m,d}- \frac{\epsilon}{2})$$
and hence $n_1> n_0$ by \eqref{1407060225}. We continue the process and define recursively (for each $m\ge 2$)
open neighborhoods $V_{m, 1}$ of some $y_{m- 1, 1}$ (with $\overline{V_{m, 1}}\subset V_{m- 1, 1}$), $\dots$,
 $V_{m, k}$ of some $y_{m- 1, k}$ (with $\overline{V_{m, k}}\subset V_{m- 1, k}$) and $n_m> n_{m -1}$ such that
\begin{equation} \label{1407060241}
\min_{1\le i\le k} d (T^{n_{m- 1}} x_i, T^{n_{m- 1}} y_{m- 1, i})> \frac{\mathbb{L}_{m,d}- \epsilon}{2}
\end{equation}
and
\begin{equation} \label{1407060245}
\max_{0\le n\le n_{m- 1}} \max_{1\le i\le k} \diam (T^n V_{m, i})\le \frac{\epsilon}{2},~ n_m\in
\bigcap_{i= 1}^k S_T (V_{m, i}, \mathbb{L}_{m,d}- \frac{\epsilon}{2}).
\end{equation}
Since by the construction, for each $i= 1, \dots, k$, $\bigcap_{m\ge 1} V_{m, i}\neq \varnothing $, we take
a point $y_i$ from the intersection (and so $y_i\in U_i$). Directly from \eqref{1407060241} and \eqref{1407060245} we have
$$\limsup_{n\rightarrow \infty} \min_{1\le i\le k} d (T^n x_i, T^n y_i)\ge \frac{\mathbb{L}_{m,d}}{2}- \epsilon.$$
Thus the conclusion follows from the arbitrariness of $\epsilon> 0$.
\end{proof}

As a consequence, we have:

\begin{prop} \label{multi-proposition}
The following conditions are
equivalent:
\begin{itemize}
\item[1.] $(X,T)$ is multi-sensitive.

\item[2.] There exists $\delta> 0$  such that
 for any finite collection $x_1, \dots, x_k$ of points in $X$ and
any system of open neighborhoods $U_i\ni x_i\ (i= 1, \dots, k)$; there exist points
 $y_i\in U_i$ and a nonnegative integer  $n$
with $ \min_{1\le i\le k} d (T^n x_i, T^n y_i)> \delta$.

\item[3.] There exists $\delta> 0$ such that
 for any finite collection $x_1, \dots, x_k$ of points in $X$ and
any system of open neighborhoods $U_i\ni x_i\ (i= 1, \dots, k)$; there exist points $y_i\in U_i$
with $\limsup_{n\rightarrow \infty}
\min_{1\le i\le k} d (T^n x_i, T^n y_i)> \delta$.

\item[4.] There exists $\delta> 0$ such that  for any finite collection $U_1, \dots, U_k$ of opene subsets of $X$,
   there exist $x_i, y_i\in U_i$ with \\
 $\limsup_{n\rightarrow \infty} \min_{1\le i\le k} d (T^n x_i, T^n y_i) > \delta$.
\end{itemize}
\end{prop}

Before proceeding, we need:

\begin{lem} \label{1407061612}
Given a dynamical system $(X, T)$, let $\delta>0$, $k\in \mathbb{N}$ and $x_i\in X$ with a neighborhood $U_i$ for each $i= 1, \dots, k$. If
$\mathbb{L}_{m,r} > \delta$, then
\begin{equation*}
\mathcal{N}= \left\{n\in \mathbb{N}: \min_{1\le i\le k} d (T^n x_i, T^n y_i)> \delta\ \text{for some}\ y_1\in U_1,
 \dots, y_k\in U_k\right\}\in \mathcal{F}_{ip}.
\end{equation*}
\end{lem}
\begin{proof}
By the assumption, $\mathcal{N}\neq \varnothing $ as $\delta< \mathbb{L}_{m,r}$. Now assume that
$$\mathcal{A}= \{p_{i_1}+ \dots+ p_{i_j}: 1\le i_1< \dots< i_j\le l\}\subset \mathcal{N}$$
for some $\{p_1, \dots, p_l\}\subset \mathbb{N}$ with $l\in \mathbb{N}$. We shall find $p_{l+ 1}\in \mathbb{N}$ such that
$p_{l+ 1}+ \mathcal{A}_0\subset \mathcal{N}$ with $\mathcal{A}_0= \{0\}\cup \mathcal{A}$,
and then obtain the conclusion by induction.

Take $x_{s, i}\in X$ with $T^{p_1+ \dots+ p_l- s} x_{s, i}= x_i$ for each $s\in \mathcal{A}_0$ and any $i= 1, \dots, k$.
 Since $\delta< \mathbb{L}_{m,r}$, obviously we can choose $q_l> p_1+ \dots+ p_l$ and $y_{s, i}\in T^{- (p_1+ \dots+ p_l- s)} U_i$
  for each $s\in \mathcal{A}_0$ and any $i= 1, \dots, k$ such that
\begin{equation} \label{1407061723}
\min_{s\in \mathcal{A}_0} \min_{1\le i\le k} d (T^{q_l} x_{s, i}, T^{q_l} y_{s, i})> \delta.
\end{equation}
Set $p_{l+ 1}= q_l- (p_1+ \dots+ p_l)\in \mathbb{N}$ and $x_{s, i}'= T^{p_1+ \dots+ p_l- s} y_{s, i}\in U_i$ for each
$s\in \mathcal{A}_0$ and any $i= 1, \dots, k$.  Then \eqref{1407061723} means equivalently
$$\min_{s\in \mathcal{A}_0} \min_{1\le i\le k} d (T^{p_{l+ 1}+ s} x_i, T^{p_{l+ 1}+ s} x_{s, i}')> \delta,$$
that is, $p_{l+ 1}+ \mathcal{A}_0\subset \mathcal{N}$, which finishes the proof.
\end{proof}

Then we have:

\begin{prop} \label{1407052219}
Let  $(X, T)$ be a system  with  a dense set of distal points. Then $\mathbb{L}_{m,r}= \overline{\mathbb{L}}_{m,r}$.
\end{prop}
\begin{proof}
By \eqref{1407052335} and Lemma \ref{1407052201} we have $2 \overline{\mathbb{L}}_{m,r}\ge \mathbb{L}_{m,r}\ge \overline{\mathbb{L}}_{m,r}$.
 Thus we only need prove $\mathbb{L}_{m,r}\le \overline{\mathbb{L}}_{m,r}$ in the case of $\mathbb{L}_{m,r}> 0$.

Let $\delta> 0$ be small enough with $\mathbb{L}_{m,r}> \delta$, and we take an open cover $\{V_1, \dots, V_p\}$ of $X$ with
$\max_{1\le i\le p} \diam (V_i)< \delta$. Now let $k\in \mathbb{N}$ and $x_i\in X$ with a neighborhood $U_i$
for each $i= 1, \dots, k$, and for each $s= (s_1, \dots, s_k)\in \{1, \dots, p\}^k$ we set
$$\mathcal{N}_s= \{n\in \mathbb{N}: T^n x_1\in V_{s_1}, \dots, T^n x_k\in V_{s_k}\}.$$
Observe from Lemma \ref{1407061612} that
\begin{equation*}
\mathcal{N}= \left\{n\in \mathbb{N}: \min_{1\le i\le k} d (T^n x_i, T^n y_i)> \mathbb{L}_{m,r}- \delta\ \text{for some}\ y_1\in U_1,
\dots, y_k\in U_k\right\}
\end{equation*}
is an IP set, and then $\mathcal{N}\cap \mathcal{N}_t$ is also an IP set for some $t\in \{1, \dots, p\}^k$,
 because $\mathcal{N}= \bigcup_{s\in \{1, \dots, p\}^k} (\mathcal{N}\cap \mathcal{N}_s)$. Choose
 $\{q_0, q_1, q_2, \dots\}\subset \mathbb{N}$ with $\{q_{i_1}+ \dots+ q_{i_j}:
 j\in \mathbb{N}\ \text{and}\ 0\le i_1< \dots< i_j\}\subset \mathcal{N}\cap \mathcal{N}_t$,
  and hence
  $q_0+ \mathcal{T}\subset \mathcal{N}\cap \mathcal{N}_t$ for some $\mathcal{T}\in \mathcal{F}_{ip}$.

Since $q_0\in \mathcal{N}$, there exists $y_i\in U_i$ for each $i= 1, \dots, k$ such that
\begin{equation} \label{1407061810}
\min_{1\le i\le k} d (T^{q_0} x_i, T^{q_0} y_i)> \mathbb{L}_{m,r}- \delta.
\end{equation}
Note that since the set of distal points is dense in $X$, we may assume that all points $y_1, \dots, y_k$ are
distal. Then all of $T^{q_0} y_1, \dots, T^{q_0} y_k$ (and hence $(T^{q_0} y_1, \dots, T^{q_0} y_k)$ in the product system $(X\times \dots\times X, T\times \dots\times T)$) are
also distal points. In particular,
$$\mathcal{M}= \left\{n\in \mathbb{N}: \max_{1\le i\le k} d (T^{q_0} y_i, T^{q_0+ n} y_i)< \delta\right\}$$
 is an IP$^*$ set \cite[Theorem 9.11]{Furstenberg1981}. Thus $\mathcal{M}\cap \mathcal{T}\neq \varnothing $, which is in fact an infinite set.
 Observing \eqref{1407061810}, it is easy to check from the construction that
 $d (T^{q_0+ r} x_i, T^{q_0+ r} y_i)> \mathbb{L}_{m,r}- 3 \delta$ for each $r\in \mathcal{M}\cap \mathcal{T}$
  and any $i= 1, \dots, k$ (as $r\in \mathcal{M}$ and $q_0, q_0+ r\in \mathcal{N}_t$). Then the conclusion follows
   from the arbitrariness of $\delta> 0$.
\end{proof}

We can not require $\mathbb{L}_{m,r}> 0$ under the assumption of
 Proposition \ref{1407052219} even for a minimal system with positive topological entropy, for example,
the above mentioned (in the previous section) Toeplitz flow with positive topological entropy.


When assume that the considered system is a transitive system, we have:

\begin{lem} \label{1407241808}
Let $(X, T)$ be a transitive system. Then $\mathbb{L}_{m,d}= \overline{\mathbb{L}}_{m,d}$.
\end{lem}
\begin{proof}
It suffices to show that $\mathbb{L}_{m,d}\le \overline{\mathbb{L}}_{m,d}$ in the case of $\mathbb{L}_{m,d}> 0$.
Let $k\in \mathbb{N}$ and take opene $U_1, \dots, U_k\subset X$. Let $\delta> 0$ be small enough with $\mathbb{L}_{m,d} > \delta$.
By the definition there exist $n\in \mathbb{N}$ and $x_i', y_i'\in U_i$ for each $i= 1, \dots,
k$ with $\min_{1\le i\le k} d (T^n x_i', T^n y_i')> \delta$. Then,
for each $i= 1, \dots, k$ we can find open $x_i'\in V_i\subset U_i$ and $y_i'\in W_i\subset U_i$ such
that both $\diam (T^n V_i)$ and $\diam (T^n W_i)$ are small enough,
thus $\min_{1\le i\le k} \dist (T^n V_i, T^n W_i)> \delta$.

Since $(X, T)$ is transitive, take $z\in \Tran (X, T)$ and then choose
$s_i, t_i\in \mathbb{N}$ with $T^{s_i} z\in V_i$ and $T^{t_i} z\in W_i$ for each $i= 1, \dots, k$.
Observe that since  $m$ belongs to $\mathbb{N}$ such that $T^m z$ is sufficiently close to $z$, we have  $T^{s_i+ m} z\in V_i$
and $T^{t_i+ m} z\in W_i$ for each $i= 1, \dots, k$, and hence
$\min_{1\le i\le k} d (T^{n+ s_i+ m} z, T^{n+ t_i+ m} z)> \delta$. Since $z\in \Tran (X, T)$,
clearly there are infinitely many $m_1< m_2< \dots$ in $\mathbb{N}$ such that each $T^{m_j} z$ is close enough
to $z$, and hence we obtain $\overline{\mathbb{L}}_{m,d}\ge \delta$ by taking $x_i= T^{s_i} z\in U_i$
and $y_i= T^{t_i} z\in U_i$ for each $i= 1, \dots, k$, finishing the proof.
\end{proof}

By the same proof of \cite[Theorem 4.1]{Kolyada2013} one has:

\begin{prop} \label{1407062303}
Let $(X, T)$ be a nontrivial weakly mixing system. Then $\mathbb{L}_{m,d}= \overline{\mathbb{L}}_{m,d}=\diam (X)$
 and $\mathbb{L}_{m,r}= \overline{\mathbb{L}}_{m,r}> 0$.
\end{prop}

\section{Further discusions}

In this section we explore more properties for multi-sensitive systems, and prove that any multi-sensitive system has positive topological sequence entropy.


As a direct corollary of \cite[Theorem 4.3]{HLSY} and the dichotomy theorems \cite[Theorem 3.1 and Proposition 3.2]{HKZ}, one has that any minimal multi-sensitive system has positive topological sequence entropy. In fact, the conclusion remains true if we remove the assumption of minimality.

\begin{prop} \label{1407091722}
Let $(X, T)$ be a multi-sensitive system. Then $(X,T)$ has positive topological sequence entropy.
\end{prop}
\begin{proof}
We are going to define an increasing sequence $\mathcal{N}= \{n_1 < n_2< \dots\}\subset \mathbb{N}$ and
a sequence of $(k+ 1,T,\epsilon)$-separated subsets of $X$ (with  respect
to $\mathcal{N}$) with cardinality $2^k, k= 1, 2, ...$ Then obviously we will have that $h_\mathcal{N} (T) \ge \log 2$.

Let $(X, T)$ be a multi-sensitive system with a sensitivity constant $2 \delta> 0$. Take  opene $U_{(1)}, U_{(2)}\subset X$ with $\dist (U_{(1)}, U_{(2)})> \delta$, and define
$V_{(1)}= T^{- n_1} U_{(1)}, V_{(2)}= T^{- n_1} U_{(2)}$ for an integer $n_1\in \mathbb{N}$. Obviously any two points $x_1\in V_{(1)}, x_2\in V_{(2)}$ are  $(2,T,\delta)$-separated.
Since $(X, T)$ is multi-sensitive with the sensitivity constant $2 \delta$, the Lyapunov number
 $\overline{\mathbb{L}}_{m,r}$ is not smaller than $\delta$ by Lemma \ref{1407052201} and hence
there exist four points $x_{(1, 1)}, x_{(1, 2)} \in V_{(1)}$ and $x_{(2, 1)}, x_{(2, 2)}\in  V_{(2)}$ and an integer $n_2 > n_1$ with
$\min_{i\in \{1, 2\}} d (T^{n_2} x_{(i, 1)}, T^{n_2} x_{(i, 2)})> \delta$.  Therefore they are  $(3,T,\delta)$-separated, as
 $\min_{i, j\in \{1, 2\}} d (T^{n_1} x_{(1, i)}, T^{n_1} x_{(2, j)})> \delta$ and $\min_{i\in \{1, 2\}} d (T^{n_2} x_{(i, 1)}, T^{n_2} x_{(i, 2)})> \delta$.
Take a small enough neighborhood $U_{(i, j)}$ of point $T^{n_2} x_{(i, j)}$
 such that $\min_{i\in \{1, 2\}} \dist (U_{(i, 1)}, U_{(i, 2)})> \delta$ and we define $V_{(i,j)}= T^{- n_2} U_{(i, j)}\cap V_{(i)}$ for  $i, j\in \{1, 2\}$.

Now assume that by induction we have defined the sequence of positive integers $n_1< \dots< n_k$ and $2^{k}$ points $x_s$ in open subsets $V_s, s\in \{1, 2\}^k$ such that
$$\min_{s, s'\in \{1, 2\}^k, s\neq s'} \max_{1\le q\le k} \dist (T^{n_q} V_s, T^{n_q} V_{s'}) > \delta,$$ and therefore
the set of points $\{x_s\in V_s, s\in \{1, 2\}^k\}$ is  $(k+ 1,T,\delta)$-separated.

Since $(X, T)$ is multi-sensitive, there exist $n_{k+ 1} > n_k$ and points $x_{(s_1, \dots, s_k, 1)} \in V_{(s_1, \dots, s_k)},~ x_{(s_1, \dots, s_k, 2)} \in V_{(s_1, \dots, s_k)}$
with
$ d (T^{n_{k+1}} x_{(s_1, \dots, s_k, 1)}, T^{n_{k+1}} x_{(s_1, \dots, s_k, 2)})> \delta$
 for any $s_1, \dots ,s_k \in \{1, 2\}^k$.
So, the set of all these $2^{k+1}$ points is  $(k+2,T,\delta)$-separated (with  respect
to $\mathcal{N}$), because by the induction hypothesis any two different points  $x_{(s_1, \dots, s_i, \dots , s_k, l)} \not
= x_{(s_1, \dots, s_i', \dots , s_k, l)}$ are also $(k+ 1,T,\delta)$-separated (with  respect to $\mathcal{N}$) for any
$l\in \{1, 2\}$.  This finishes the proof.
\end{proof}


  To link multi-sensitivity of a system with local
 equicontinuity of points in the system, in \cite{HKZ} the first, third and fourth authors of the present paper introduced the concept of
 syndetically equicontinuous points of a system.
 Recall that $\text{Eq}_{\text{syn}} (X, T)$ denotes the set of all syndetically equicontinuous points of $(X, T)$.
  Since a thick set has a nonempty intersection with a syndetic set, one has readily that
 if $(X, T)$ is
 thickly sensitive then $\text{Eq}_{\text{syn}} (X, T)= \varnothing$.
  By the dichotomy theorem \cite[Theorem 3.5]{HKZ}, one has that:
a transitive system is either multi-sensitive or contains syndetically equicontinuous points, and a minimal system is either multi-sensitive or its each point is syndetically equicontinuous.

\begin{prop} \label{1407182012}
Consider the following conditions for a dynamical system $(X, T)$:
\begin{enumerate}

\item $(X, T)$ is equicontinous.

\item For every $\epsilon> 0$ there exist a $\delta> 0$ and a syndetic subset $\mathcal{A}\subset \mathbb{N}$ such that, for any $x,y \in X$, $d (x, y)< \delta$ implies $d (T^n x, T^n y)<
\epsilon$ for all $n\in \mathcal{A}$.

\item $\text{Eq}_{\text{syn}} (X, T)= X$.

\item For every $\epsilon> 0$ there exists a $\delta> 0$ such that, for every $U\subset X$ with $\text{diam} (U)< \delta$ there exists a syndetic subset $\mathcal{A}\subset \mathbb{N}$ such that, $x,y \in U$ implies $d (T^n x, T^n y)<
\epsilon$ for all $n\in \mathcal{A}$.

\item For every $\epsilon> 0$ there exist a $\delta> 0$ and an $m\in \mathbb{N}$ such  that, for any $x,y\in X$,
$d (x, y)< \delta$ implies $\min_{0\le i\le m}d (T^{n+ i} x, T^{n+ i} y)<
\epsilon$ for all $n\in \mathbb{N}$.
\end{enumerate}
Then $(1)\Longleftrightarrow (2)\Longrightarrow (3)\Longleftrightarrow (4)\Longrightarrow (5)$.
\end{prop}
\begin{proof}
It suffices to prove $(2)\Longrightarrow (1)$, $(3)\Longrightarrow (4)$ and $(3)\Longrightarrow (5)$.


\noindent $(2)\Longrightarrow (1)$:
Let $\epsilon> 0$. By the condition (2) there exist $\epsilon> \delta> 0$ and syndetic $\mathcal{A}\subset \mathbb{N}$ such that $d (x, y)< \delta$ implies $d (T^n x, T^n y)<
\epsilon$ for any $x,y \in X$ and  $n\in \mathcal{A}$. Since syndetic sets have ``bounded gaps", there exists $m\in \mathbb{N}$ such
that $\{n+ i: n\in A, i= 0, 1, \dots, m\}\supset \{m+ 1, m+ 2, \dots\}$.  Therefore there exists $\delta'> 0$ such  that, for any $x,y \in X$ and $i= 0, 1, \dots, m$, $d (x, y)< \delta'$
implies $d (T^i x, T^i y)< \delta< \epsilon$ (hence, by the selection of $\delta$, $d (T^{n+ i} x, T^{n+ i} y)< \epsilon$ for all $n\in \mathcal{A}$).
So, $d (x, y)< \delta'$ implies $d (T^j x, T^j y)< \epsilon$ for any $x,y \in X$ and $j\in \mathbb{N}$, i.e., $(X, T)$ is equicontinuous.


\noindent $(3)\Longrightarrow (4)$:
Let $\epsilon> 0$. Since $\text{Eq}_{\text{syn}} (X, T)= X$, for any $x\in X$ there exist open $U_x\subset X$ containing $x$ and a syndetic set $\mathcal{A}_x\subset \mathbb{N}$ such
that  $d (T^n x, T^n x')< \epsilon$ for all $x'\in U_x$ whenever $n\in \mathcal{A}_x$ (and hence $d (T^n x', T^n x'')< 2 \epsilon$ whenever $x', x''\in U_x$). Observe that $X$ is a compact metric space, we can take a set of points $\{x_1, \dots, x_s\} \subset X$ such that $\{U_{x_j}: j= 1, \dots, s\}$ forms
an open cover of $X$. Then there exists $\delta> 0$ such that, every subset $U\subset X$ with $\text{diam} (U)< \delta$ is contained in some  $U_{x_j}$, and then
$d (T^n x, T^n y)< 2 \epsilon$ whenever $x, y\in U$ and $n\in \mathcal{A}_{x_j}$.


\noindent $(3)\Longrightarrow (5)$: The proof is similar to that of $(3)\Longrightarrow (4)$.
Let $\epsilon> 0$. For any $x\in X$, the open set $U_x\subset X$ containing $x$, the syndetic set $\mathcal{A}_x\subset \mathbb{N}$ and the set $\{x_1, \dots, x_s\} \subset X$ are constructed as in the proof of $(3)\Longrightarrow (4)$. We
take $m_x\in \mathbb{N}$ such that $\{n, n+ 1, \dots, n+ m_x\}\cap \mathcal{A}_x\neq \varnothing$ for each $n\in \mathbb{N}$, and therefore
$\min_{0\le i\le m_x}d (T^{n+ i} x', T^{n+ i} x'')<
2 \epsilon$ for any $x', x''\in U_x$ and $n\in \mathbb{N}$. Set $m= \max \{m_{x_j}: j= 1, \dots, s\}$.
By the construction there exists $\delta> 0$ such that any points $x,y\in X$ with $d (x, y)< \delta$ are contained in some $U_{x_j}$, and so for each $n\in \mathbb{N}$,
$\min_{0\le i\le m}d (T^{n+ i} x, T^{n+ i} y)\le \min_{0\le i\le m_{x_j}}d (T^{n+ i} x, T^{n+ i} y)< 2 \epsilon$.
\end{proof}

Remark that by \cite[Example 3.6]{HKZ} there is a sensitive transitive non-minimal system $(X, T)$ with $\text{Eq}_{\text{syn}} (X, T)= X$.
In fact, it is not hard to see the difference between the conditions (2), (4) and (5). Precisely, given the parameters $\epsilon$ and $\delta$:
\begin{enumerate}

\item
 in the condition (2) the syndetic subset $\mathcal{A}\subset \mathbb{N}$ is independent of all $x, y\in X$ with $d (x, y)< \delta$;

\item
in the condition (4) the syndetic subset $\mathcal{A}\subset \mathbb{N}$ depends on every subset $U\subset X$ with $\text{diam} (U)< \delta$;

\item
the condition (5) is equivalent to say, for all $x, y\in X$ with $d (x, y)< \delta$, there exists a syndetic subset $\mathcal{A}_{x, y}\subset \mathbb{N}$ (with a uniform bound $m\in \mathbb{N}$ for the gaps) such that $d (T^n x, T^n y)<
\epsilon$ for all $n\in \mathcal{A}_{x, y}$, here the syndetic subset $\mathcal{A}_{x, y}$ depends on $x$ and $y$ with $d (x, y)< \delta$.
\end{enumerate}


At the end of this section, following the proof of \cite[Lemma 4.1]{HKZ} we give a sufficient condition for its each point being syndetically equicontinuous.

\begin{lem} \label{1407070139}
Let $\pi: (X, T)\rightarrow (Y, S)$ be a factor map, where $(Y, S)$ is a minimal equicontinuous system. Assume that $y_0$ belongs to $Y$, such that $\pi^{- 1} (y_0)$
is a singleton. Then $\text{Eq}_{\text{syn}} (X, T)= X$.
\end{lem}
\begin{proof}
We choose a compatible metric $\rho$ over $Y$, and recall that $d$ is the metric over $X$. Let $x\in X$ and $\delta> 0$. Say $\{x_0\}= \pi^{- 1} (y_0)$. We take open $U_0\subset X$ containing $x_0$ such that $\text{diam} (U_0)< \delta$, and then take $\delta_1> 0$ such that $\pi^{- 1} (V_0)\subset U_0$ where $V_0= \{y\in Y: \rho (y_0, y)< 2 \delta_1\}$. As $(Y, S)$ is equicontinuous, there exists $\delta_1\ge \epsilon_1> 0$ such that $\rho (S^n y_1, S^n y_2)< \delta_1$ whenever $\rho (y_1, y_2)< \epsilon_1$ and $n\in \mathbb{N}$.
Set $V_1= \{y\in Y: \rho (y_0, y)< \delta_1\}\subset V_0$. By the minimality of  $(Y, S)$, there exists $m\in \mathbb{N}$ such that $S^m (\pi x)\in V_1$, and
then put $\mathcal{S}= N_S (S^m (\pi x), V_1)$ which is a syndetic set.

Now take open $U_1\subset X$ containing $x$ such that $\text{diam} (\pi U_1)< \epsilon_1$.
Let $n\in \mathcal{S}$ and $x'\in U_1$. Then $S^{n+ m} (\pi x)\in V_1$, and $\rho (\pi x, \pi x')< \epsilon_1$ and so $\rho (S^{n+ m} \pi x, S^{n+ m} \pi x')< \delta_1$. Thus $\pi (T^{n+ m} x')= S^{n+ m} (\pi x')\in V_0$, and then $T^{n+ m} x'\in U_0$ by the construction of $V_0$.
In particular, $d (T^{n+ m} x, T^{n+ m} x')< \delta$.
That is, $x\in \text{Eq}_{\text{syn}} (X, T)$, because $m+ \mathcal{S}\subset \mathbb{N}$ is a syndetic set. This finishes the proof.
\end{proof}

\section{Related properties and open questions}

A family $\mathcal{F}$ is \emph{proper} if it is a proper subset
of $\mathcal{P}$, i.e., neither empty nor all of $\mathcal{P}$.
If
a proper family $\mathcal{F}$ satisfies $\mathcal{F}\cdot \mathcal{F}\subset \mathcal{F}$, then it is called a \emph{filter}, where
$\mathcal{F}\cdot \mathcal{F}$ is defined as $\{F_1\cap F_2: F_1, F_2\in \mathcal{F}\}$.
Recall that in \cite{Furstenberg1967} H. Furstenberg has showed that a dynamical system $(X,T)$ is weakly mixing if and only if $\mathcal{N}_T$ is
a filter.
 The inverse problem --- when a filter $\mathcal{F}$  can be realized via  $\mathcal{N}_T$ for a dynamical system $(X,T)$,
is very hard (see, for instance \cite{ShYe2004}).

For $\delta> 0$, denote by $\mathcal{S}_T (\delta)$
the set of all subsets of $\mathbb{Z}_+$ containing $S_T (U, \delta)$ for some opene subset $U$ of $X$; and denote by $\mathcal{S}_T$ the set of all subsets of $\mathbb{Z}_+$ which come from $\mathcal{S}_T (\delta)$ for some $\delta> 0$.
 It seems that this problem is even much harder for a filter realization via
$\mathcal{S}_T (\delta)$ for a sensitive dynamical system $(X,T)$ and some $\delta> 0$. So, let $(X,T)$
be a weakly mixing system and $\delta < \diam(X)$. Then $S_T (U, \delta)\supset N_T (U, V_1)\cap N_T (U, V_2)\supset N_T (U', V')$ for some opene subsets $U', V'$ of $X$ by \cite{Furstenberg1967}, when opene subsets $U, V_1, V_2$ satisfy
$\dist (V_1, V_2)> \delta$, and so $\mathcal{S}_T (\delta) \subset \mathcal{N}_T$.
We say that a dynamical system $(X,T)$ is \emph{filter-sensitive} if there is $\delta >0$ such that $\mathcal{S}_T (\delta)$ is a filter. So, obviously that
a topologically mixing system is filter-sensitive, and one can show that any filter-sensitive system is thick sensitive.  We do not know if they are
equivalent for the transitive case, in particular, we do not know if any weakly mixing system is filter-sensitive. We see here at least two problems with realization. Take two elements $S_T(U,\delta)$ and $S_T(V,\delta)$
of
$\mathcal{S}_T (\delta)$ of a sensitive dynamical system $(X,T)$. When the intersection $S_T(U,\delta) \cap S_T(V,\delta)$ is again in $\mathcal{S}_T (\delta)$? The second question is
if there is a realization of a filter-sensitive system? Similar question is also with following situation. Let $\mathcal{F}$ be the family of all thickly syndetic
subsets, which can be easily checked to be a filter. What about the realization this family via $\mathcal{S}_T (\delta)$
for a thickly syndetically sensitive system $(X,T)$?

V{\'{\i}}ctor Jim{\'e}nez~L{\'o}pez and L'ubom{\'{\i}}r Snoha introduced and studied the notion of the (\emph{Misiurewicz})
\emph{stroboscopical property} in \cite{JiSn2003}. One can study the stroboscopical property via Furstenberg families. Let $\mathcal{F}$ be a (Furstenberg) family.
We will say a system $(X,T)$ satisfies \emph{$\mathcal{F}$-stroboscopical property} if for any $A\in \mathcal{F}$ and $z\in X$ there is a point
$x\in X$ with $z\in \omega_T(A,x)$, where $\omega_T(A,x)$ is the set of all limit points of the set $T^A(x)=\{T^i x: i\in A\}$. Moreover, we will
say that the system has \emph{strongly $\mathcal{F}$-stroboscopical property} if any $A\in \mathcal{F}$ and $z\in X$ the set of such
$x$ is dense.

Recall that a system $(X,T)$ is \emph{$\mathcal{F}$-transitive} if $N_T(U,V)\in \mathcal{F}$ for any opene sets $U,V \subset X$.
If a  system $(X,T)$ is  $\mathcal{F}$-transitive, then not so hard to show (say similarly as it was done in Lemma \ref{02000}) that for any  $F \in k\mathcal{F}$ and opene sets
$U,V$ in $X$ the set $F\cap N_T(U,V)$ is infinite.

\begin{prop} The following conditions are equivalent:
\begin{enumerate}

\item[(1)] $(X,T)$ has strongly $\mathcal{F}$-stroboscopical property.

\item[(2)] $(X,T)$ is $k\mathcal{F}$-transitive.
\end{enumerate}
\end{prop}

\begin{proof} (1)$\Rightarrow$ (2) Assume that the system is not $k\mathcal{F}$-transitive. Then there is a pair of opene sets
$U,V$ for which $N_T(U,V)\notin k\mathcal{F}$. Take $F\in \mathcal{F}$ with $N_T(U,V)\cap F=\varnothing$. Hence
$T^nU\cap V=\varnothing$ for all $n\in F$. Since it, for any $z\in V$ there is no point $x\in U$ with  $z\in \omega_T(F,x)$.
A contradiction to the assumption.

(2)$\Rightarrow$ (1) Fix $z\in X, F\in \mathcal{F}$ and an opene $G\subset X$. We are going to find $x\in G$ with $z\in \omega_T(F,x)$.
Let $B_n$ be the opene ball of radius $\frac{1}{n}$ centred at $z$ and $G_1=G$. Then  $N_T(G_1,B_1)\in k\mathcal{F}$ and there is $n_1\in F$
 such that $T^{n_1}G_1\cap B_1\neq \varnothing$. Take opene $G_2\subset \overline{G_2}\subset G_1\cap T^{-n_1}B_1$ and proceed by induction.

 Suppose we have defined $G_k$. Since $N_T(G_k,B_k)\cap F=\varnothing$ is infinite, we can choose in $F$ an integer $n_k >n_{k-1}$ with
 $T^{n_k}G_k\cap B_k\neq \varnothing$. Finally, let $G_{k+1}$ be an opene set in $X$ such set $G_{k+1}\subset \overline{G_{k+1}}\subset G_k
 \cap T^{-n_k}B_k$.

 We have found a sequence $\{ n_k\}_{k=1}^\infty \in F$ such that $T^{n_{k- 1}}G_k \subset B_{k-1}$ for $k \geq 2$. So, we have $\varnothing
 \neq \bigcap_{k=1}^\infty \overline{G_k} = \bigcap_{k=1}^\infty G_k$, because $\overline{G_k}$ are nested closed sets. Therefore there is
 $x\in \bigcap_{k=1}^\infty G_k$ and $T^{n_k}x \in B_k$ for any $k$. Thus $z\in \omega_T(F,x)$.
\end{proof}

We end this section with the following questions:

\medskip



\noindent \textbf{Question 1.} If $(X,T)$ is a weakly mixing system, then does it follow that
for any point $x\in X$ there exists a  point $y\in X$
such that $\omega_T(y) = \omega_{\mathcal{N}_T}(x)$?

\medskip

\noindent \textbf{Question 2.}  Are all non-minimal M-systems Li-Yorke sensitive?

\medskip

\bibliographystyle{amsplain}


\end{document}